\documentclass{article}
\usepackage[usenames]{color}
\usepackage{graphicx} % used Scott
\usepackage{comment}
\usepackage{amscd}
\usepackage[english]{babel}
\usepackage{color}
\usepackage{fullpage}
\usepackage{psfig}
\usepackage{graphics,amsmath,amssymb}
\usepackage{amsthm}
\usepackage{amsfonts}
\usepackage{latexsym}
\usepackage{epsf}
\usepackage{float}
\usepackage{tikz}

  \usepackage[utf8]{inputenc}
  \usepackage[T1]{fontenc}
  \usepackage{geometry}
  \usepackage{wasysym}
  \usepackage{pgf}

\usepackage{colortbl}
\usepackage{pgfplots}
\pgfplotsset{compat=newest}

\usepackage[colorlinks=true,
linkcolor=webgreen,
filecolor=webbrown,
citecolor=webgreen]{hyperref}

\definecolor{webgreen}{rgb}{0,.5,0}
\definecolor{webbrown}{rgb}{.6,0,0}

  \usetikzlibrary{math,positioning,shapes,shadows,arrows,calc,intersections,through,backgrounds,matrix,mindmap}
  \usepackage{pgf}
\usepackage{colortbl}
\usepackage{pgfplots}
\pgfplotsset{compat=newest}

\definecolor{c1}{gray}{0.8}

\newcommand{\sC}{{\mathcal C}}
\newcommand{\sE}{{\mathcal E}}

\newcommand{\sN}{{\mathcal N}}

\theoremstyle{plain}
\numberwithin{equation}{section}
\newtheorem{thm}{Theorem}[section]

\newtheorem{conj}[thm]{Conjecture}
\newtheorem{problem}[thm]{Open Problem}

 \newcommand{\seqnum}[1]{\href{http://oeis.org/#1}{\underline{#1}}}  % used

\newcommand{\eqn}[1]{(\ref{#1})} % used
\newcommand{\beql}[1]{\begin{equation}\label{#1}} % used
\newcommand{\eeq}{\end{equation}}  % used

\begin{document}

%\fancyhead{}
%\renewcommand{\headrulewidth}{0pt}
%\fancyfoot{}
%\fancyfoot[LE,RO]{\medskip \thepage}
%\fancyfoot[LO]{\medskip MONTH YEAR}
%\fancyfoot[RE]{\medskip VOLUME , NUMBER }

\setcounter{page}{1}

%%%%%%%%%%%%%%%%%%%%%%%%%%%%%%%%%%%%%%
%                             Title, authors
%%%%%%%%%%%%%%%%%%%%%%%%%%%%%%%%%%%%%%

\begin{center}
{\large\bf Graphical Enumeration and Stained Glass Windows, 1: Rectangular Grids} \\
%{\large\bf Graphical Enumeration Problems and Stained Glass Windows, 1: Rectangular Grids} \\
%{\large\bf Graphical Enumeration Problems Arising From Planar Grids and Stained Glass Windows} \\
\vspace*{+.2in}

Lars Blomberg \\
\"{A}renprisv\"{a}gen 111,
SE-58564 Linghem, SWEDEN \\
Email:  \href{mailto:larsl.blomberg@comhem.se}{\tt larsl.blomberg@comhem.se}

\vspace*{+.1in}

Scott R. Shannon \\
P.O. Box 2260, Rowville, Victoria 3178, AUSTRALIA \\
Email: \href{mailto:scott\_r\_shannon@hotmail.com}{\tt scott\_r\_shannon@hotmail.com}

\vspace*{+.1in}

N. J. A. Sloane\footnote{To whom correspondence should be addressed.} \\
The OEIS Foundation Inc., 
11 South Adelaide Ave.,
Highland Park, NJ 08904, USA \\
Email:  \href{mailto:njasloane@gmail.com}{\tt njasloane@gmail.com}
\vspace*{+.1in}

DEDICATED TO THE MEMORY OF RONALD LEWIS GRAHAM (1935--2020)

%\vspace*{+.1in}
%September 16, 2020

\vspace*{+.1in}

\begin{abstract}
A survey of enumeration problems arising from the study of graphs formed
when the edges of a polygon are marked with evenly spaced points
 and every pair of  points
is joined by a line. A few of these problems have
been solved, a classical example being the 
the graph $K_n$ formed when all pairs of vertices
of a regular $n$-gon are joined by chords, which was analyzed
by Poonen and Rubinstein in 1998.
Most of these problems are unsolved, however, and this
two-part article provides data from a number of such problems as well as
colored illustrations, which are often reminiscent of stained glass windows.
The polygons considered include rectangles,  hollow rectangles (or frames), 
 triangles, pentagons, pentagrams, crosses, etc., as well
as figures formed by drawing semicircles  joining equally-spaced points on a line.
%The paper ends with a brief discussion of the problem of how to
 %design aesthetically pleasing colorings for these graphs.
 
This first part discusses rectangular grids. The $1 \times n$ grids, or equally the
graphs $K_{n+1,n+1}$, were studied by Legendre and Griffiths, and here we 
investigate the number of cells with a given number of edges and the number of nodes with a 
given degree. We have only partial results for the $m \times n$ rectangles, including 
upper bounds on the numbers of nodes and cells.
 \end{abstract}
\end{center}

%\maketitle

%%%%%%%%%%%%%%%%%%%%%%%%%%%%%%%%%%%%%%
% Section 1: Introduction
%%%%%%%%%%%%%%%%%%%%%%%%%%%%%%%%%%%%%%

\section{Introduction.}\label{Sec1}
In 1998 Poonen and Rubinstein \cite{PoRu98} (see also \cite{Som98}) solved the problem of finding 
the numbers of intersection points and cells in a regular drawing of the complete
graph $K_n$, and in 2009-2010 Legendre~\cite{Leg09} and Griffiths~\cite{Gri10}
solved a similar problem for the complete
bipartite graph $K_{n,n}$.
Stated another way, \cite{PoRu98} analyzes the graph formed by joining
all pairs of vertices of a regular $n$-gon, while~\cite{Leg09, Gri10} 
analyze the graph formed by taking a row of $n-1$ identical squares and 
drawing lines between every pair of boundary nodes.

One motivation for the present work was to see if these investigations could be
extended to graphs formed from other structures, such as an $m \times n$ array of identical squares.
%The latter problem can be restated as follows.
Take a rectangle of size $m \times n$, and place $m-1$ equally spaced points on 
the two vertical sides,
and $n-1$ equally spaced points on the two horizontal sides, and then draw lines between every pair of the $2(m+n)$ boundary 
points. The resulting planar graph, which we denote by $BC(m,n)$, is the main subject of Part~1 of this paper.

Although we have not been very successful in analyzing these graphs, we have collected a great deal of data,
which has been entered into various sequences in the
\emph{On-Line Encyclopedia of Integer Sequences}~\cite{OEIS}.

In Part 2~\cite{Rose2}, we continue this work by considering other structures such as hollow squares
(or ``frames''), triangles, pentagons, hexagons, pentagrams, etc., as well  
as figures formed by drawing semicircles  joining equally-spaced points on an interval.
The last-mentioned figures are reminiscent of juggling patterns,\footnote{See entry 
\seqnum{A290447} in~\cite{OEIS}}
as studied by Ron Graham in ~\cite{RLG1} and other papers, and we regret that now it is too late to ask him for help for finding  a formula for those numbers.

We were also motivated  by memories of stained glass
 windows seen in the great Gothic cathedrals of Northern Europe.
 In 2019 we made a colored drawing of $K_{23}$ (Fig.~\ref{FigK23}) 
 which was reminiscent of a rose window, and we were curious to see what colored versions of 
 other graphs would look like. Informally, our philosophy has been, if we can't solve it, make art.
 We make no great claims for artistic merit, but the images are certainly colorful.

 %%%%%%%%%%%%%%%%%%%%%%%%%%%%%%%%%%%%%%%%%%%%%%%
 %%%%   FIG 1
 %%%%%%%%%%%%%%%%%%%%%%%%%%%%%%%%%%%%%%%%%%%%%%
 % The Poonen-Rubinstein seqs are A007678 (which has the pictures), A006561, A006600
 \begin{figure}[!ht]
\centerline{\includegraphics[angle=0, width=6.4in]{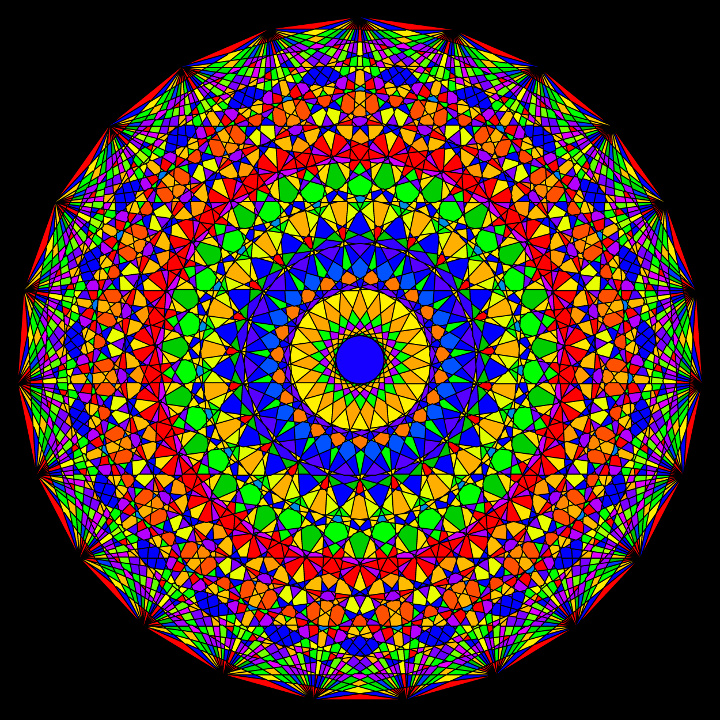}}
\caption{Colored drawing of complete graph $K_{23}$ (see 
  \textsection\ref{SubSecRan} for coloring algorithm).
  Entry \seqnum{A007678} in~\cite{OEIS} has many similar images (which are also of higher quality). }\label{FigK23}
\end{figure}

Space limitations have restricted the number and quality of the images
that we could include here. The corresponding 
entries in \cite{OEIS} (\seqnum{A007678}\footnote{Six-digit numbers prefixed by A refer to entries in the On-Line Encyclopedia of Integer Sequences \cite{OEIS}.}
 in the case of Fig.~\ref{FigK23})
contain a large number of other images, with better resolution. We are especially fond of the three images
of $K_{41}$ in \seqnum{A007678}, and in \seqnum{A331452}  the reader should not miss 
the images labeled $T(10,2)$, $T(6,6)$, $T(7,7)$, which are drawings of the graphs
$BC(10,2)$, $BC(6,6)$, and $BC(7,7)$ discussed below.

This paper is arranged as follows. The last section of this Introduction establishes the notation we will use,
especially the terms \emph{nodes}, \emph{chords}, and \emph{cells}, and
provides some examples. Section~\ref{SecBC1n} deals with the graphs $BC(1,n)$ (or equivalently $BC(n,1)$),
where the underlying polygon is a rectangle of size $1 \times n$ (or $n \times 1$).
Theorem~\ref{ThVRBC1}  gives Legendre and Griffiths's enumeration of the nodes
and cells in $BC(1,n)$.  In 2019, Max Alekseyev (personal communication; see also \seqnum{A306302})
 pointed out that that the Legendre-Griffiths results are essentially the same as results that
he and his coauthors obtained in connection with the enumeration of two-dimensional  threshold functions
\cite{Alex10, ABZ15}. 
The family of isosceles triangle graphs $IT(n)$  (Section~\ref{SecIT})  provides a bridge
between the graphs $BC(1,n)$ and two-dimensional threshold function.
Alekseyev also mentioned that their work implies a result 
that was apparently overlooked in the Legendre and Griffiths papers: the cells in $BC(1,n)$ are always triangles or quadrilaterals. See Theorem~\ref{ThABZ}.
The proof of this fact in \cite{ABZ15} depends on a theorem
 about teaching sets for threshold function~\cite{SZ98, Zol01}.
 We feel that such a elementary property should have a purely geometrical proof, although no such proof is presently known.
 We state this question as Open Problem~\ref{OP1}.

One possible attack on this problem is to study the distribution of
cells in each of the $n$ squares of $BC(1,n)$--- see
Tables~\ref{TabWangt}, \ref{TabWangq}, \ref{TabWangc}.
The  \emph{gfun} Maple program \cite{GFUN} suggests a form for the generating
functions of the columns of these tables, but so far this is only a conjecture.

We next consider the number of interior nodes in $BC(1,n)$ where $c$ chords meet 
(Table~\ref{TabWangi}). The number of simple nodes, where just two chords cross, is of the greatest 
interest, since these seem to dominate. But even though we have calculated $500$ terms
of this sequence (Table~\ref{TabA334701}, \seqnum{A334701}) we have been unable to find a formula 
or recurrence (Open Problem~\ref{OPSIP}).
There have been several similar occasions  during this project when we have regretted not having an oracle that 
would take a few hundred terms of a simple, well-defined sequence and 
suggest some kind of formula.\footnote{The oracle  might compare the sequence with a
shifted version of each of the  $300000$
entries in \cite{OEIS}, and ask
Bruno Salvy and Paul Zimmermann's program \emph{gfun},
or Harm Derksen's program \emph{guesss}, or 
Christian Kratthentaler's program \emph{Rate},
or one of the other programs used by \emph{Superseeker} \cite{Slo10} 
 if there is a formula for the difference.}

The graph $BC(1,n)$ starts from a $1 \times n$ rectangle. 
If we start from an $m \times n$ rectangle, with $m$ and $n > 1$, there are actually three natural ways
to define a graph, which we will denote by
$BC(m,n)$, $AC(m,n)$, and $LC(m,n)$. These are the subjects of Sections \ref{SecBCmn},
\ref{SecAC}, and \ref{SecLC}, respectively. For these families we have plenty of data 
and pictures, but not many results.  
In Section~\ref{SecBCmn} we conjecture that the
cells in $BC(2,n)$ have at most eight sides, and for $n \ge 19$, at most six sides
(Conjecture~\ref{ConjBC2cell}).
Our main result concerning $BC(m,n)$  is
an upper bound on the numbers of nodes and cells in $BC(m,n)$,
presented in \textsection\ref{SecBCGP}, which appears to be reasonably close to the true values. 

 The final section (\textsection\ref{SecColor}) describes the algorithms that were used to color the graphs.

\vspace{.2in}

%%%%%%%%%%%%%%%%%%%%%%%%%%%%%%%%%%%%%%
%%%%%%%%%%%%%%%%%%%         FIG 2
%%%%%%%%%%%%%%%%%%%%%%%%%%%%%%%%%%%%%%
%\begin{figure}[!ht]
%   \centerline{\includegraphics[angle=0, width=2.0in]{SC.1.1.bwGimp.png}}
 %  \caption{$BC(1,1)$: a $1$-reticulated square with four cells.} 
% \end{figure}

 \begin{figure}[!ht] % Fig 2
\begin{center}
\begin{tikzpicture}[scale=1.5]
\draw (0,0) -- (2,0);
\draw (0,0) -- (0,2);
\draw (2,2) -- (2,0);
\draw (2,2) -- (0,2);
\draw (0,0) -- (2,2);
\draw (0,2) -- (2,0);
code
\end{tikzpicture}
\caption{$BC(1,1)$: a $1$-reticulated square with four cells.}
\label{FigBC11bw}
\end{center}
\end{figure}
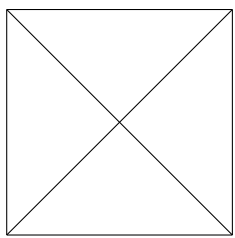

 \vspace*{+.2in}
 \noindent{\bf Terminology.}
The graphs that we study are usually constructed by 
starting with a \emph{polygon} $P$ drawn in the plane, having \emph{sides} and  \emph{vertices}.
We then subdivide the sides by dividing them into some number
of equal parts. To divide a side into $k$ equal parts, we insert $k-1$ 
equally spaced \emph{nodes} along that side.  
The side then contains the  two end-vertices and the $k-1$ internal nodes.
We say that the side has been  \emph{$k$-reticulated}. 
A \emph{chord} in $P$  is a finite line-segment joining a pair of vertices or nodes.
 A chord is undirected, does not extend outside $P$, and has specified end points.

Given a polygon $P$, we construct a planar graph $G$ by drawing 
chords according to some specified rule. 
For example, we might join every vertex or node to all the vertices or nodes on all the other sides.
Points where chords intersect are also nodes of the graph $G$.
The graph thus formed has \emph{nodes} (referring to the vertices and edge-nodes of
the polygon and also any interior intersection points),
\emph{edges} (which are line segments between pairs of nodes), 
and \emph{cells} (the connected regions defined by the edges). 
In graph theory the cells are  sometimes called \emph{faces} or \emph{chambers}, but we will not use those terms.
Our graphs are also \emph{maps} in the sense of Tutte~\cite{Tut63, Tut68}, but we will refer to
them simply as planar graphs.

For a connected planar graph, Euler's formula states that the numbers of nodes, edges, 
and cells are related by
\beql{Eq1}
|\, \mbox{nodes}\,| ~-~ |\,\mbox{edges}\,|  ~+~ |\,\mbox{cells}\,| ~=~ 1.
\eeq

Examples: Figure~\ref{FigBC11bw} shows the graph
$BC(1,1)$ (defined in \textsection\ref{SecBC1n}),
which has $5$ nodes, $8$ edges, and $4$ cells. The polygon is a     
square and there are two chords which meet at the central node. 
%  a $1$-reticulated square with four vertices and two chords, 
Figure~\ref{FigBC22bw} shows the graph $BC(2,2)$, constructed from a square in
which each side has been $2$-reticulated. There are $14$ chords.
The graph has $37$ nodes,  $92$ edges, and $56$ cells.\footnote{ $BC(3,3)$ is
shown in Fig.~\ref{FigBC33} in ~\textsection\ref{SecBCmn} and has $340$ cells. 
There is no known formula  for the number of cells in $BC(n,n)$, even though we have $52$ terms.
The sequence begins $4, 56, 340,  1120,3264, \ldots$ (\seqnum{A255011}).} Figure~\ref{FigBC22} shows a colored version of
the graph. The principles used to color these graphs are discussed in Section~\ref{SecColor}.
For any undefined terms from graph theory see \cite{Bela, Harary}. 

%%%%%%%%%%%%%%%%%%%%%%%%%%%%%%%%%%%%%%
%%%%%%%%%%%%%%%%%%%         FIG 3
%%%%%%%%%%%%%%%%%%%%%%%%%%%%%%%%%%%%%%

\begin{figure}[!ht]
        \begin{minipage}[b]{0.45\linewidth}
                % Start of Fig 3 minipage
                %\begin{figure}[!ht]
                \begin{center}
                        \begin{tikzpicture}[scale=2.45]
                        \draw (0,0) -- (2,0);
                        \draw (0,0) -- (2,1);
                        \draw (0,0) -- (2,2);
                        \draw (0,0) -- (1,2);
                        \draw (0,0) -- (0,2);
                        \draw (0,1) -- (1,0);
                        \draw (0,1) -- (2,0);
                        \draw (0,1) -- (2,1);
                        \draw (0,1) -- (2,2);
                        \draw (0,1) -- (1,2);
                        \draw (0,2) -- (2,0);
                        \draw (0,2) -- (1,0);
                        \draw (0,2) -- (2,1);
                        \draw (0,2) -- (2,2);
                        \draw (1,0) -- (1,2);
                        \draw (1,0) -- (2,2);
                        \draw (1,0) -- (2,1);
                        \draw (1,2) -- (2,1);
                        \draw (1,2) -- (2,0);
                        \draw (2,0) -- (2,2);
                        \end{tikzpicture}
                        \caption{$BC(2,2)$: a $2$-reticulated square with 56 cells.}
                        \label{FigBC22bw}
                \end{center}
                %\end{figure}
                % \centerline{\includegraphics[angle=0, width=0.7\textwidth]{SC.2.2.bwGimp.png}}
                %
                %\label{FigBC22bw}
                % End of Fig 3 minipage
        \end{minipage}
        \hspace{0.4cm}
        %%%%%%%%%%%%%%%%%%%%%%%%%%%%%%%%%%%%%%
        %%%%%%%%%%%%%%%%%%%%%%%%%%%%%%%%%%%%%%
        %%%%%%%%%%%%%%%%%%%         FIG 4
        % Source in OEIS, entry A331452:
        %     %H Scott R. Shannon, <a href="/A331452/a331452_12.png">Colored illustration for T(2,2)</a>
        %%%%%%%%%%%%%%%%%%%%%%%%%%%%%%%%%%%%%%
        \begin{minipage}[b]{0.45\linewidth}
                %%\centerline{\includegraphics[angle=0, width=0.7\textwidth]{SC.2.2.sGimp2.png}}
%\begin{center}
\begin{tikzpicture}[scale=0.243,black,semithick,line join=round]

\definecolor{c0}{RGB}{179,125,0}
\definecolor{c1}{RGB}{153,97,0}
\definecolor{c2}{RGB}{203,156,0}
\definecolor{c3}{RGB}{102,0,0}
\definecolor{c4}{RGB}{225,188,0}
\definecolor{c5}{RGB}{245,223,0}

\draw [fill=c0] (0.000,0.000) -- (10.000,0.000) -- (6.667,3.333) -- (6.667,3.333) -- cycle;
\draw [fill=c1] (0.000,0.000) -- (6.667,3.333) -- (5.000,5.000) -- (5.000,5.000) -- cycle;
\draw [fill=c1] (0.000,0.000) -- (5.000,5.000) -- (3.333,6.667) -- (3.333,6.667) -- cycle;
\draw [fill=c0] (0.000,0.000) -- (3.333,6.667) -- (0.000,10.000) -- (0.000,10.000) -- cycle;
\draw [fill=c1] (0.000,10.000) -- (3.333,6.667) -- (4.000,8.000) -- (4.000,8.000) -- cycle;
\draw [fill=c2] (0.000,10.000) -- (4.000,8.000) -- (5.000,10.000) -- (5.000,10.000) -- cycle;
\draw [fill=c2] (0.000,10.000) -- (5.000,10.000) -- (4.000,12.000) -- (4.000,12.000) -- cycle;
\draw [fill=c1] (0.000,10.000) -- (4.000,12.000) -- (3.333,13.333) -- (3.333,13.333) -- cycle;
\draw [fill=c0] (0.000,10.000) -- (3.333,13.333) -- (0.000,20.000) -- (0.000,20.000) -- cycle;
\draw [fill=c1] (0.000,20.000) -- (3.333,13.333) -- (5.000,15.000) -- (5.000,15.000) -- cycle;
\draw [fill=c1] (0.000,20.000) -- (5.000,15.000) -- (6.667,16.667) -- (6.667,16.667) -- cycle;
\draw [fill=c0] (0.000,20.000) -- (6.667,16.667) -- (10.000,20.000) -- (10.000,20.000) -- cycle;
\draw [fill=c3] (3.333,6.667) -- (5.000,5.000) -- (6.667,6.667) -- (4.000,8.000) -- (4.000,8.000) -- cycle;
\draw [fill=c3] (3.333,13.333) -- (4.000,12.000) -- (6.667,13.333) -- (5.000,15.000) -- (5.000,15.000) -- cycle;
\draw [fill=c4] (4.000,8.000) -- (6.667,6.667) -- (5.000,10.000) -- (5.000,10.000) -- cycle;
\draw [fill=c4] (4.000,12.000) -- (5.000,10.000) -- (6.667,13.333) -- (6.667,13.333) -- cycle;
\draw [fill=c3] (5.000,5.000) -- (6.667,3.333) -- (8.000,4.000) -- (6.667,6.667) -- (6.667,6.667) -- cycle;
\draw [fill=c5] (5.000,10.000) -- (10.000,10.000) -- (6.667,13.333) -- (6.667,13.333) -- cycle;
\draw [fill=c5] (5.000,10.000) -- (6.667,6.667) -- (10.000,10.000) -- (10.000,10.000) -- cycle;
\draw [fill=c3] (5.000,15.000) -- (6.667,13.333) -- (8.000,16.000) -- (6.667,16.667) -- (6.667,16.667) -- cycle;
\draw [fill=c1] (6.667,3.333) -- (10.000,0.000) -- (8.000,4.000) -- (8.000,4.000) -- cycle;
\draw [fill=c4] (6.667,6.667) -- (8.000,4.000) -- (10.000,5.000) -- (10.000,5.000) -- cycle;
\draw [fill=c5] (6.667,6.667) -- (10.000,5.000) -- (10.000,10.000) -- (10.000,10.000) -- cycle;
\draw [fill=c4] (6.667,13.333) -- (10.000,15.000) -- (8.000,16.000) -- (8.000,16.000) -- cycle;
\draw [fill=c5] (6.667,13.333) -- (10.000,10.000) -- (10.000,15.000) -- (10.000,15.000) -- cycle;
\draw [fill=c1] (6.667,16.667) -- (8.000,16.000) -- (10.000,20.000) -- (10.000,20.000) -- cycle;
\draw [fill=c2] (8.000,4.000) -- (10.000,0.000) -- (10.000,5.000) -- (10.000,5.000) -- cycle;
\draw [fill=c2] (8.000,16.000) -- (10.000,15.000) -- (10.000,20.000) -- (10.000,20.000) -- cycle;
\draw [fill=c0] (10.000,0.000) -- (20.000,0.000) -- (13.333,3.333) -- (13.333,3.333) -- cycle;
\draw [fill=c1] (10.000,0.000) -- (13.333,3.333) -- (12.000,4.000) -- (12.000,4.000) -- cycle;
\draw [fill=c2] (10.000,0.000) -- (12.000,4.000) -- (10.000,5.000) -- (10.000,5.000) -- cycle;
\draw [fill=c5] (10.000,5.000) -- (13.333,6.667) -- (10.000,10.000) -- (10.000,10.000) -- cycle;
\draw [fill=c4] (10.000,5.000) -- (12.000,4.000) -- (13.333,6.667) -- (13.333,6.667) -- cycle;
\draw [fill=c5] (10.000,10.000) -- (15.000,10.000) -- (13.333,13.333) -- (13.333,13.333) -- cycle;
\draw [fill=c5] (10.000,10.000) -- (13.333,13.333) -- (10.000,15.000) -- (10.000,15.000) -- cycle;
\draw [fill=c5] (10.000,10.000) -- (13.333,6.667) -- (15.000,10.000) -- (15.000,10.000) -- cycle;
\draw [fill=c2] (10.000,15.000) -- (12.000,16.000) -- (10.000,20.000) -- (10.000,20.000) -- cycle;
\draw [fill=c4] (10.000,15.000) -- (13.333,13.333) -- (12.000,16.000) -- (12.000,16.000) -- cycle;
\draw [fill=c1] (10.000,20.000) -- (12.000,16.000) -- (13.333,16.667) -- (13.333,16.667) -- cycle;
\draw [fill=c0] (10.000,20.000) -- (13.333,16.667) -- (20.000,20.000) -- (20.000,20.000) -- cycle;
\draw [fill=c3] (12.000,4.000) -- (13.333,3.333) -- (15.000,5.000) -- (13.333,6.667) -- (13.333,6.667) -- cycle;
\draw [fill=c3] (12.000,16.000) -- (13.333,13.333) -- (15.000,15.000) -- (13.333,16.667) -- (13.333,16.667) -- cycle;

\draw [fill=c1] (13.333,3.333) -- (20.000,0.000) -- (15.000,5.000) -- (15.000,5.000) -- cycle;
\draw [fill=c4] (13.333,6.667) -- (16.000,8.000) -- (15.000,10.000) -- (15.000,10.000) -- cycle;
\draw [fill=c3] (13.333,6.667) -- (15.000,5.000) -- (16.667,6.667) -- (16.000,8.000) -- (16.000,8.000) -- cycle;
\draw [fill=c4] (13.333,13.333) -- (15.000,10.000) -- (16.000,12.000) -- (16.000,12.000) -- cycle;
\draw [fill=c3] (13.333,13.333) -- (16.000,12.000) -- (16.667,13.333) -- (15.000,15.000) -- (15.000,15.000) -- cycle;
\draw [fill=c1] (13.333,16.667) -- (15.000,15.000) -- (20.000,20.000) -- (20.000,20.000) -- cycle;
\draw [fill=c1] (15.000,5.000) -- (20.000,0.000) -- (16.667,6.667) -- (16.667,6.667) -- cycle;
\draw [fill=c2] (15.000,10.000) -- (20.000,10.000) -- (16.000,12.000) -- (16.000,12.000) -- cycle;
\draw [fill=c2] (15.000,10.000) -- (16.000,8.000) -- (20.000,10.000) -- (20.000,10.000) -- cycle;
\draw [fill=c1] (15.000,15.000) -- (16.667,13.333) -- (20.000,20.000) -- (20.000,20.000) -- cycle;
\draw [fill=c1] (16.000,8.000) -- (16.667,6.667) -- (20.000,10.000) -- (20.000,10.000) -- cycle;
\draw [fill=c1] (16.000,12.000) -- (20.000,10.000) -- (16.667,13.333) -- (16.667,13.333) -- cycle;
\draw [fill=c0] (16.667,6.667) -- (20.000,0.000) -- (20.000,10.000) -- (20.000,10.000) -- cycle;
\draw [fill=c0] (16.667,13.333) -- (20.000,10.000) -- (20.000,20.000) -- (20.000,20.000) -- cycle;
\draw [thick] (0.000,0.000) -- (20.000,0.000) -- (20.000,20.000) -- (0.000,20.000) -- (0.000,20.000) -- cycle;
\end{tikzpicture}
%\end{center}
\caption{ The same $BC(2,2)$ drawn  with colored cells. See \textsection\ref{SubSecYR} for coloring scheme.}
                \label{FigBC22}
        \end{minipage}
\end{figure}

%%%%%%%%%%%%%%%%%%%%%%%%%%%%%%%%%%%%%%
% Section 2: BC(1,n) - used to be called SC(1,n)  -- Rectangles
%%%%%%%%%%%%%%%%%%%%%%%%%%%%%%%%%%%%%%

\section{\texorpdfstring{$BC(1,n)$: $1 \times n$}{BC(1,n): 1 X n} rectangular windows}\label{SecBC1n}
The graph $BC(1,n)$ ($n \ge 1$) is constructed by taking a $1 \times n$ rectangle, 
inserting $n-1$ equally spaced nodes
along the top and bottom sides, and then joining every pair of vertices or nodes by chords.
Figures~\ref{FigBC11bw}, \ref{FigBC12}, and \ref{FigBC13}
% \ref{FigBC14} 
show $BC(1,n)$ for $n = 1, 2$, and $3$.

\begin{figure}[!ht]
        %%%%%%%%%%%%%%%%%%%%%%%%%%%%%%%%%%%%%%
        %%%%%%%%%%%%%%%%%%%         FIG 5
        % Source in OEIS, entry A331452:
        %     %H Scott R. Shannon, <a href="/A331452/a331452_7.png">Colored illustration for T(2,1)</a>
        %%%%%%%%%%%%%%%%%%%%%%%%%%%%%%%%%%%%%%
        \begin{minipage}[b]{0.40\linewidth}
        %%      \centerline{\includegraphics[angle=90, width=0.7\textwidth]{SC.2.1sGimp.png}}

                \begin{tikzpicture}[scale=0.24,black,semithick,line join=round]
                %\begin{tikzpicture}[scale=0.300,black,semithick,line join=round]

                \definecolor{c0}{RGB}{185,133,0}
                \definecolor{c1}{RGB}{153,97,0}
                \definecolor{c2}{RGB}{215,172,0}
                \definecolor{c3}{RGB}{102,0,0}
                \definecolor{c4}{RGB}{240,214,0}

                \draw [fill=c0] (0.000,0.000) -- (10.000,0.000) -- (6.667,3.333) -- (6.667,3.333) -- cycle;
                \draw [fill=c1] (0.000,0.000) -- (6.667,3.333) -- (5.000,5.000) -- (5.000,5.000) -- cycle;
                \draw [fill=c2] (0.000,0.000) -- (5.000,5.000) -- (0.000,10.000) -- (0.000,10.000) -- cycle;
                \draw [fill=c1] (0.000,10.000) -- (5.000,5.000) -- (6.667,6.667) -- (6.667,6.667) -- cycle;
                \draw [fill=c0] (0.000,10.000) -- (6.667,6.667) -- (10.000,10.000) -- (10.000,10.000) -- cycle;
                \draw [fill=c3] (5.000,5.000) -- (6.667,3.333) -- (10.000,5.000) -- (6.667,6.667) -- (6.667,6.667) -- cycle;
                \draw [fill=c4] (6.667,3.333) -- (10.000,0.000) -- (10.000,5.000) -- (10.000,5.000) -- cycle;
                \draw [fill=c4] (6.667,6.667) -- (10.000,5.000) -- (10.000,10.000) -- (10.000,10.000) -- cycle;
                \draw [fill=c0] (10.000,0.000) -- (20.000,0.000) -- (13.333,3.333) -- (13.333,3.333) -- cycle;
                \draw [fill=c4] (10.000,0.000) -- (13.333,3.333) -- (10.000,5.000) -- (10.000,5.000) -- cycle;
                \draw [fill=c4] (10.000,5.000) -- (13.333,6.667) -- (10.000,10.000) -- (10.000,10.000) -- cycle;
                \draw [fill=c3] (10.000,5.000) -- (13.333,3.333) -- (15.000,5.000) -- (13.333,6.667) -- (13.333,6.667) -- cycle;
                \draw [fill=c0] (10.000,10.000) -- (13.333,6.667) -- (20.000,10.000) -- (20.000,10.000) -- cycle;
                \draw [fill=c1] (13.333,3.333) -- (20.000,0.000) -- (15.000,5.000) -- (15.000,5.000) -- cycle;
                \draw [fill=c1] (13.333,6.667) -- (15.000,5.000) -- (20.000,10.000) -- (20.000,10.000) -- cycle;
                \draw [fill=c2] (15.000,5.000) -- (20.000,0.000) -- (20.000,10.000) -- (20.000,10.000) -- cycle;
                \draw [thick] (0.000,0.000) -- (20.000,0.000) -- (20.000,10.000) -- (0.000,10.000) -- (0.000,10.000) -- cycle;
                \end{tikzpicture}
 \caption{$BC(1,2)$.}
 %\caption{$BC(1,2)$, based on a $1 \times 2$ array of squares.}
                \label{FigBC12}
        \end{minipage}
\hspace{0.1cm}
        %%%%%%%%%%%%%%%%%%%%%%%%%%%%%%%%%%%%%%
        %%%%%%%%%%%%%%%%%%%         FIG 6
        % Source in OEIS, entry A331452:
        %     %H Scott R. Shannon, <a href="/A331452/a331452_8.png">Colored illustration for T(3,1)</a>
        %%%%%%%%%%%%%%%%%%%%%%%%%%%%%%%%%%%%%%
        \begin{minipage}[b]{0.45\linewidth}
%%              \centerline{\includegraphics[angle=90, width=0.84\textwidth]{SC.3.1sGimp.png}}
        \begin{tikzpicture}[scale=0.24,black,semithick,line join=round]
      %  \begin{tikzpicture}[scale=0.300,black,semithick,line join=round]

        \definecolor{c0}{RGB}{169,114,0}
        \definecolor{c1}{RGB}{153,97,0}
        \definecolor{c2}{RGB}{215,172,0}
        \definecolor{c3}{RGB}{240,214,0}
        \definecolor{c4}{RGB}{176,0,0}
        \definecolor{c5}{RGB}{138,0,0}
        \definecolor{c6}{RGB}{251,237,0}
        \definecolor{c7}{RGB}{185,133,0}
        \definecolor{c8}{RGB}{200,152,0}
        \definecolor{c9}{RGB}{102,0,0}
        \definecolor{c10}{RGB}{228,193,0}
        \definecolor{c11}{RGB}{217,0,0}

        \draw [fill=c0] (0.000,0.000) -- (10.000,0.000) -- (7.500,2.500) -- (7.500,2.500) -- cycle;
        \draw [fill=c1] (0.000,0.000) -- (7.500,2.500) -- (6.667,3.333) -- (6.667,3.333) -- cycle;
        \draw [fill=c2] (0.000,0.000) -- (6.667,3.333) -- (5.000,5.000) -- (5.000,5.000) -- cycle;
        \draw [fill=c3] (0.000,0.000) -- (5.000,5.000) -- (0.000,10.000) -- (0.000,10.000) -- cycle;
        \draw [fill=c2] (0.000,10.000) -- (5.000,5.000) -- (6.667,6.667) -- (6.667,6.667) -- cycle;
        \draw [fill=c1] (0.000,10.000) -- (6.667,6.667) -- (7.500,7.500) -- (7.500,7.500) -- cycle;
        \draw [fill=c0] (0.000,10.000) -- (7.500,7.500) -- (10.000,10.000) -- (10.000,10.000) -- cycle;
        \draw [fill=c4] (5.000,5.000) -- (6.667,3.333) -- (10.000,5.000) -- (6.667,6.667) -- (6.667,6.667) -- cycle;
        \draw [fill=c5] (6.667,3.333) -- (7.500,2.500) -- (10.000,3.333) -- (10.000,5.000) -- (10.000,5.000) -- cycle;
        \draw [fill=c5] (6.667,6.667) -- (10.000,5.000) -- (10.000,6.667) -- (7.500,7.500) -- (7.500,7.500) -- cycle;
        \draw [fill=c6] (7.500,2.500) -- (10.000,0.000) -- (10.000,3.333) -- (10.000,3.333) -- cycle;
        \draw [fill=c6] (7.500,7.500) -- (10.000,6.667) -- (10.000,10.000) -- (10.000,10.000) -- cycle;
        \draw [fill=c7] (10.000,0.000) -- (20.000,0.000) -- (15.000,2.500) -- (15.000,2.500) -- cycle;
        \draw [fill=c8] (10.000,0.000) -- (15.000,2.500) -- (13.333,3.333) -- (13.333,3.333) -- cycle;
        \draw [fill=c5] (10.000,0.000) -- (13.333,3.333) -- (12.000,4.000) -- (10.000,3.333) -- (10.000,3.333) -- cycle;
        \draw [fill=c6] (10.000,3.333) -- (12.000,4.000) -- (10.000,5.000) -- (10.000,5.000) -- cycle;
        \draw [fill=c6] (10.000,5.000) -- (12.000,6.000) -- (10.000,6.667) -- (10.000,6.667) -- cycle;
        \draw [fill=c9] (10.000,5.000) -- (12.000,4.000) -- (15.000,5.000) -- (12.000,6.000) -- (12.000,6.000) -- cycle;
        \draw [fill=c5] (10.000,6.667) -- (12.000,6.000) -- (13.333,6.667) -- (10.000,10.000) -- (10.000,10.000) -- cycle;
        \draw [fill=c8] (10.000,10.000) -- (13.333,6.667) -- (15.000,7.500) -- (15.000,7.500) -- cycle;
        \draw [fill=c7] (10.000,10.000) -- (15.000,7.500) -- (20.000,10.000) -- (20.000,10.000) -- cycle;
        \draw [fill=c10] (12.000,4.000) -- (13.333,3.333) -- (15.000,5.000) -- (15.000,5.000) -- cycle;
        \draw [fill=c10] (12.000,6.000) -- (15.000,5.000) -- (13.333,6.667) -- (13.333,6.667) -- cycle;
        \draw [fill=c11] (13.333,3.333) -- (15.000,2.500) -- (16.667,3.333) -- (15.000,5.000) -- (15.000,5.000) -- cycle;
        \draw [fill=c11] (13.333,6.667) -- (15.000,5.000) -- (16.667,6.667) -- (15.000,7.500) -- (15.000,7.500) -- cycle;
        \draw [fill=c8] (15.000,2.500) -- (20.000,0.000) -- (16.667,3.333) -- (16.667,3.333) -- cycle;
        \draw [fill=c10] (15.000,5.000) -- (18.000,6.000) -- (16.667,6.667) -- (16.667,6.667) -- cycle;
        \draw [fill=c10] (15.000,5.000) -- (16.667,3.333) -- (18.000,4.000) -- (18.000,4.000) -- cycle;
        \draw [fill=c9] (15.000,5.000) -- (18.000,4.000) -- (20.000,5.000) -- (18.000,6.000) -- (18.000,6.000) -- cycle;
        \draw [fill=c8] (15.000,7.500) -- (16.667,6.667) -- (20.000,10.000) -- (20.000,10.000) -- cycle;
        \draw [fill=c5] (16.667,3.333) -- (20.000,0.000) -- (20.000,3.333) -- (18.000,4.000) -- (18.000,4.000) -- cycle;
        \draw [fill=c5] (16.667,6.667) -- (18.000,6.000) -- (20.000,6.667) -- (20.000,10.000) -- (20.000,10.000) -- cycle;
        \draw [fill=c6] (18.000,4.000) -- (20.000,3.333) -- (20.000,5.000) -- (20.000,5.000) -- cycle;
        \draw [fill=c6] (18.000,6.000) -- (20.000,5.000) -- (20.000,6.667) -- (20.000,6.667) -- cycle;
        \draw [fill=c0] (20.000,0.000) -- (30.000,0.000) -- (22.500,2.500) -- (22.500,2.500) -- cycle;
        \draw [fill=c6] (20.000,0.000) -- (22.500,2.500) -- (20.000,3.333) -- (20.000,3.333) -- cycle;
        \draw [fill=c5] (20.000,3.333) -- (22.500,2.500) -- (23.333,3.333) -- (20.000,5.000) -- (20.000,5.000) -- cycle;
        \draw [fill=c5] (20.000,5.000) -- (23.333,6.667) -- (22.500,7.500) -- (20.000,6.667) -- (20.000,6.667) -- cycle;
        \draw [fill=c4] (20.000,5.000) -- (23.333,3.333) -- (25.000,5.000) -- (23.333,6.667) -- (23.333,6.667) -- cycle;
        \draw [fill=c6] (20.000,6.667) -- (22.500,7.500) -- (20.000,10.000) -- (20.000,10.000) -- cycle;
        \draw [fill=c0] (20.000,10.000) -- (22.500,7.500) -- (30.000,10.000) -- (30.000,10.000) -- cycle;
        \draw [fill=c1] (22.500,2.500) -- (30.000,0.000) -- (23.333,3.333) -- (23.333,3.333) -- cycle;
        \draw [fill=c1] (22.500,7.500) -- (23.333,6.667) -- (30.000,10.000) -- (30.000,10.000) -- cycle;
        \draw [fill=c2] (23.333,3.333) -- (30.000,0.000) -- (25.000,5.000) -- (25.000,5.000) -- cycle;
        \draw [fill=c2] (23.333,6.667) -- (25.000,5.000) -- (30.000,10.000) -- (30.000,10.000) -- cycle;
        \draw [fill=c3] (25.000,5.000) -- (30.000,0.000) -- (30.000,10.000) -- (30.000,10.000) -- cycle;
        \draw [thick] (0.000,0.000) -- (30.000,0.000) -- (30.000,10.000) -- (0.000,10.000) -- (0.000,10.000) -- cycle;
        \end{tikzpicture}
        \caption{$BC(1,3)$. }\label{FigBC13}
 %\caption{$BC(1,3)$, based on a $1 \times 3$ array of squares. (See  \textsection\ref{SubSecYR} for explanation of colors.) }\label{FigBC13}
%\caption{Just a test}\label{FigBC13again}
%\caption{$BC(1,3)$.}\label{FigBC13}
                %, based on a $1 \times 3$ array of squares. (See  \textsection\ref{SubSecYR} for explanation of colors.) }
                \end{minipage}
\end{figure}

%%%% Lars's Fig 6 starts here

%%%%%%%%%%%%%%%%%%%%%%%%%%%%%%%%%%%%%%
%%%%%%%%%%%%%%%%%%%         FIG 7
% Source in OEIS, entry A331452:
%     %H Scott R. Shannon, <a href="/A331452/a331452_9.png">Colored illustration for T(4,1)</a>
%%%%%%%%%%%%%%%%%%%%%%%%%%%%%%%%%%%%%%
\begin{figure}[!ht]
        %%\centerline{\includegraphics[angle=0, width=1.3in]{SC.4.1s2Gimp2.png}}

% start of minipage 7
\begin{minipage}[b]{0.45\linewidth}
\begin{center}
        \begin{tikzpicture}[scale=0.300,black,semithick,line join=round]
\definecolor{c0}{RGB}{246,225,0}
\definecolor{c1}{RGB}{215,172,0}
\definecolor{c2}{RGB}{169,114,0}
\definecolor{c3}{RGB}{153,97,0}
\definecolor{c4}{RGB}{185,133,0}
\definecolor{c5}{RGB}{251,237,0}
\definecolor{c6}{RGB}{234,203,0}
\definecolor{c7}{RGB}{207,162,0}
\definecolor{c8}{RGB}{161,105,0}
\definecolor{c9}{RGB}{193,142,0}
\definecolor{c10}{RGB}{200,152,0}
\definecolor{c11}{RGB}{255,248,0}
\definecolor{c12}{RGB}{130,0,0}
\definecolor{c13}{RGB}{226,0,0}
\definecolor{c14}{RGB}{209,0,0}
\definecolor{c15}{RGB}{116,0,0}
\definecolor{c16}{RGB}{161,0,0}
\definecolor{c17}{RGB}{176,0,0}
\definecolor{c18}{RGB}{243,0,0}
\definecolor{c19}{RGB}{240,214,0}
\definecolor{c20}{RGB}{228,193,0}
\definecolor{c21}{RGB}{192,0,0}
\definecolor{c22}{RGB}{145,0,0}
\definecolor{c23}{RGB}{221,182,0}
\definecolor{c24}{RGB}{177,123,0}
\definecolor{c25}{RGB}{102,0,0}

\draw [fill=c0] (0.000,0.000) -- (10.000,0.000) -- (5.000,5.000) -- (5.000,5.000) -- cycle;
\draw [fill=c1] (0.000,0.000) -- (5.000,5.000) -- (3.333,6.667) -- (3.333,6.667) -- cycle;
\draw [fill=c2] (0.000,0.000) -- (3.333,6.667) -- (2.500,7.500) -- (2.500,7.500) -- cycle;
\draw [fill=c3] (0.000,0.000) -- (2.500,7.500) -- (2.000,8.000) -- (2.000,8.000) -- cycle;
\draw [fill=c4] (0.000,0.000) -- (2.000,8.000) -- (0.000,10.000) -- (0.000,10.000) -- cycle;
\draw [fill=c5] (0.000,10.000) -- (2.000,8.000) -- (2.500,10.000) -- (2.500,10.000) -- cycle;
\draw [fill=c6] (0.000,10.000) -- (2.500,10.000) -- (3.333,13.333) -- (3.333,13.333) -- cycle;
\draw [fill=c7] (0.000,10.000) -- (3.333,13.333) -- (2.500,15.000) -- (2.500,15.000) -- cycle;
\draw [fill=c8] (0.000,10.000) -- (2.500,15.000) -- (2.000,16.000) -- (2.000,16.000) -- cycle;
\draw [fill=c9] (0.000,10.000) -- (2.000,16.000) -- (0.000,20.000) -- (0.000,20.000) -- cycle;
\draw [fill=c10] (0.000,20.000) -- (2.000,16.000) -- (2.500,17.500) -- (2.500,17.500) -- cycle;
\draw [fill=c11] (0.000,20.000) -- (2.500,17.500) -- (3.333,20.000) -- (3.333,20.000) -- cycle;
\draw [fill=c11] (0.000,20.000) -- (3.333,20.000) -- (2.500,22.500) -- (2.500,22.500) -- cycle;
\draw [fill=c10] (0.000,20.000) -- (2.500,22.500) -- (2.000,24.000) -- (2.000,24.000) -- cycle;
\draw [fill=c9] (0.000,20.000) -- (2.000,24.000) -- (0.000,30.000) -- (0.000,30.000) -- cycle;
\draw [fill=c8] (0.000,30.000) -- (2.000,24.000) -- (2.500,25.000) -- (2.500,25.000) -- cycle;
\draw [fill=c7] (0.000,30.000) -- (2.500,25.000) -- (3.333,26.667) -- (3.333,26.667) -- cycle;
\draw [fill=c6] (0.000,30.000) -- (3.333,26.667) -- (2.500,30.000) -- (2.500,30.000) -- cycle;
\draw [fill=c5] (0.000,30.000) -- (2.500,30.000) -- (2.000,32.000) -- (2.000,32.000) -- cycle;
\draw [fill=c4] (0.000,30.000) -- (2.000,32.000) -- (0.000,40.000) -- (0.000,40.000) -- cycle;
\draw [fill=c3] (0.000,40.000) -- (2.000,32.000) -- (2.500,32.500) -- (2.500,32.500) -- cycle;
\draw [fill=c2] (0.000,40.000) -- (2.500,32.500) -- (3.333,33.333) -- (3.333,33.333) -- cycle;
\draw [fill=c1] (0.000,40.000) -- (3.333,33.333) -- (5.000,35.000) -- (5.000,35.000) -- cycle;
\draw [fill=c0] (0.000,40.000) -- (5.000,35.000) -- (10.000,40.000) -- (10.000,40.000) -- cycle;
\draw [fill=c12] (2.000,8.000) -- (2.500,7.500) -- (3.333,10.000) -- (2.500,10.000) -- (2.500,10.000) -- cycle;
\draw [fill=c13] (2.000,16.000) -- (2.500,15.000) -- (3.333,16.667) -- (2.500,17.500) -- (2.500,17.500) -- cycle;
\draw [fill=c13] (2.000,24.000) -- (2.500,22.500) -- (3.333,23.333) -- (2.500,25.000) -- (2.500,25.000) -- cycle;
\draw [fill=c12] (2.000,32.000) -- (2.500,30.000) -- (3.333,30.000) -- (2.500,32.500) -- (2.500,32.500) -- cycle;
\draw [fill=c14] (2.500,7.500) -- (3.333,6.667) -- (5.000,10.000) -- (3.333,10.000) -- (3.333,10.000) -- cycle;
\draw [fill=c15] (2.500,10.000) -- (3.333,10.000) -- (4.000,12.000) -- (3.333,13.333) -- (3.333,13.333) -- cycle;
\draw [fill=c16] (2.500,15.000) -- (3.333,13.333) -- (4.000,16.000) -- (3.333,16.667) -- (3.333,16.667) -- cycle;
\draw [fill=c17] (2.500,17.500) -- (3.333,16.667) -- (4.000,18.000) -- (3.333,20.000) -- (3.333,20.000) -- cycle;
\draw [fill=c17] (2.500,22.500) -- (3.333,20.000) -- (4.000,22.000) -- (3.333,23.333) -- (3.333,23.333) -- cycle;
\draw [fill=c16] (2.500,25.000) -- (3.333,23.333) -- (4.000,24.000) -- (3.333,26.667) -- (3.333,26.667) -- cycle;
\draw [fill=c15] (2.500,30.000) -- (3.333,26.667) -- (4.000,28.000) -- (3.333,30.000) -- (3.333,30.000) -- cycle;
\draw [fill=c14] (2.500,32.500) -- (3.333,30.000) -- (5.000,30.000) -- (3.333,33.333) -- (3.333,33.333) -- cycle;
\draw [fill=c18] (3.333,6.667) -- (5.000,5.000) -- (6.667,6.667) -- (5.000,10.000) -- (5.000,10.000) -- cycle;
\draw [fill=c11] (3.333,10.000) -- (5.000,10.000) -- (4.000,12.000) -- (4.000,12.000) -- cycle;
\draw [fill=c19] (3.333,13.333) -- (5.000,15.000) -- (4.000,16.000) -- (4.000,16.000) -- cycle;
\draw [fill=c20] (3.333,13.333) -- (4.000,12.000) -- (5.000,15.000) -- (5.000,15.000) -- cycle;
\draw [fill=c21] (3.333,16.667) -- (4.000,16.000) -- (4.286,17.143) -- (4.000,18.000) -- (4.000,18.000) -- cycle;
\draw [fill=c11] (3.333,20.000) -- (5.000,20.000) -- (4.000,22.000) -- (4.000,22.000) -- cycle;
\draw [fill=c11] (3.333,20.000) -- (4.000,18.000) -- (5.000,20.000) -- (5.000,20.000) -- cycle;
\draw [fill=c21] (3.333,23.333) -- (4.000,22.000) -- (4.286,22.857) -- (4.000,24.000) -- (4.000,24.000) -- cycle;
\draw [fill=c19] (3.333,26.667) -- (4.000,24.000) -- (5.000,25.000) -- (5.000,25.000) -- cycle;
\draw [fill=c20] (3.333,26.667) -- (5.000,25.000) -- (4.000,28.000) -- (4.000,28.000) -- cycle;
\draw [fill=c11] (3.333,30.000) -- (4.000,28.000) -- (5.000,30.000) -- (5.000,30.000) -- cycle;
\draw [fill=c18] (3.333,33.333) -- (5.000,30.000) -- (6.667,33.333) -- (5.000,35.000) -- (5.000,35.000) -- cycle;
\draw [fill=c22] (4.000,12.000) -- (5.000,10.000) -- (6.000,12.000) -- (5.000,15.000) -- (5.000,15.000) -- cycle;
\draw [fill=c23] (4.000,16.000) -- (5.000,15.000) -- (4.286,17.143) -- (4.286,17.143) -- cycle;
\draw [fill=c24] (4.000,18.000) -- (4.286,17.143) -- (5.000,20.000) -- (5.000,20.000) -- cycle;
\draw [fill=c24] (4.000,22.000) -- (5.000,20.000) -- (4.286,22.857) -- (4.286,22.857) -- cycle;
\draw [fill=c23] (4.000,24.000) -- (4.286,22.857) -- (5.000,25.000) -- (5.000,25.000) -- cycle;
\draw [fill=c22] (4.000,28.000) -- (5.000,25.000) -- (6.000,28.000) -- (5.000,30.000) -- (5.000,30.000) -- cycle;
\draw [fill=c25] (4.286,17.143) -- (5.000,15.000) -- (5.714,17.143) -- (5.000,20.000) -- (5.000,20.000) -- cycle;
\draw [fill=c25] (4.286,22.857) -- (5.000,20.000) -- (5.714,22.857) -- (5.000,25.000) -- (5.000,25.000) -- cycle;
\draw [fill=c1] (5.000,5.000) -- (10.000,0.000) -- (6.667,6.667) -- (6.667,6.667) -- cycle;
\draw [fill=c11] (5.000,10.000) -- (6.667,10.000) -- (6.000,12.000) -- (6.000,12.000) -- cycle;
\draw [fill=c14] (5.000,10.000) -- (6.667,6.667) -- (7.500,7.500) -- (6.667,10.000) -- (6.667,10.000) -- cycle;
\draw [fill=c23] (5.000,15.000) -- (6.000,16.000) -- (5.714,17.143) -- (5.714,17.143) -- cycle;
\draw [fill=c20] (5.000,15.000) -- (6.000,12.000) -- (6.667,13.333) -- (6.667,13.333) -- cycle;
\draw [fill=c19] (5.000,15.000) -- (6.667,13.333) -- (6.000,16.000) -- (6.000,16.000) -- cycle;
\draw [fill=c11] (5.000,20.000) -- (6.667,20.000) -- (6.000,22.000) -- (6.000,22.000) -- cycle;
\draw [fill=c24] (5.000,20.000) -- (6.000,22.000) -- (5.714,22.857) -- (5.714,22.857) -- cycle;
\draw [fill=c24] (5.000,20.000) -- (5.714,17.143) -- (6.000,18.000) -- (6.000,18.000) -- cycle;
\draw [fill=c11] (5.000,20.000) -- (6.000,18.000) -- (6.667,20.000) -- (6.667,20.000) -- cycle;
\draw [fill=c20] (5.000,25.000) -- (6.667,26.667) -- (6.000,28.000) -- (6.000,28.000) -- cycle;
\draw [fill=c23] (5.000,25.000) -- (5.714,22.857) -- (6.000,24.000) -- (6.000,24.000) -- cycle;
\draw [fill=c19] (5.000,25.000) -- (6.000,24.000) -- (6.667,26.667) -- (6.667,26.667) -- cycle;
\draw [fill=c14] (5.000,30.000) -- (6.667,30.000) -- (7.500,32.500) -- (6.667,33.333) -- (6.667,33.333) -- cycle;
\draw [fill=c11] (5.000,30.000) -- (6.000,28.000) -- (6.667,30.000) -- (6.667,30.000) -- cycle;
\draw [fill=c1] (5.000,35.000) -- (6.667,33.333) -- (10.000,40.000) -- (10.000,40.000) -- cycle;
\draw [fill=c21] (5.714,17.143) -- (6.000,16.000) -- (6.667,16.667) -- (6.000,18.000) -- (6.000,18.000) -- cycle;
\draw [fill=c21] (5.714,22.857) -- (6.000,22.000) -- (6.667,23.333) -- (6.000,24.000) -- (6.000,24.000) -- cycle;
\draw [fill=c15] (6.000,12.000) -- (6.667,10.000) -- (7.500,10.000) -- (6.667,13.333) -- (6.667,13.333) -- cycle;
\draw [fill=c16] (6.000,16.000) -- (6.667,13.333) -- (7.500,15.000) -- (6.667,16.667) -- (6.667,16.667) -- cycle;
\draw [fill=c17] (6.000,18.000) -- (6.667,16.667) -- (7.500,17.500) -- (6.667,20.000) -- (6.667,20.000) -- cycle;
\draw [fill=c17] (6.000,22.000) -- (6.667,20.000) -- (7.500,22.500) -- (6.667,23.333) -- (6.667,23.333) -- cycle;
\draw [fill=c16] (6.000,24.000) -- (6.667,23.333) -- (7.500,25.000) -- (6.667,26.667) -- (6.667,26.667) -- cycle;
\draw [fill=c15] (6.000,28.000) -- (6.667,26.667) -- (7.500,30.000) -- (6.667,30.000) -- (6.667,30.000) -- cycle;
\draw [fill=c2] (6.667,6.667) -- (10.000,0.000) -- (7.500,7.500) -- (7.500,7.500) -- cycle;
\draw [fill=c12] (6.667,10.000) -- (7.500,7.500) -- (8.000,8.000) -- (7.500,10.000) -- (7.500,10.000) -- cycle;
\draw [fill=c6] (6.667,13.333) -- (7.500,10.000) -- (10.000,10.000) -- (10.000,10.000) -- cycle;
\draw [fill=c7] (6.667,13.333) -- (10.000,10.000) -- (7.500,15.000) -- (7.500,15.000) -- cycle;
\draw [fill=c13] (6.667,16.667) -- (7.500,15.000) -- (8.000,16.000) -- (7.500,17.500) -- (7.500,17.500) -- cycle;
\draw [fill=c11] (6.667,20.000) -- (10.000,20.000) -- (7.500,22.500) -- (7.500,22.500) -- cycle;
\draw [fill=c11] (6.667,20.000) -- (7.500,17.500) -- (10.000,20.000) -- (10.000,20.000) -- cycle;
\draw [fill=c13] (6.667,23.333) -- (7.500,22.500) -- (8.000,24.000) -- (7.500,25.000) -- (7.500,25.000) -- cycle;
\draw [fill=c6] (6.667,26.667) -- (10.000,30.000) -- (7.500,30.000) -- (7.500,30.000) -- cycle;
\draw [fill=c7] (6.667,26.667) -- (7.500,25.000) -- (10.000,30.000) -- (10.000,30.000) -- cycle;
\draw [fill=c12] (6.667,30.000) -- (7.500,30.000) -- (8.000,32.000) -- (7.500,32.500) -- (7.500,32.500) -- cycle;
\draw [fill=c2] (6.667,33.333) -- (7.500,32.500) -- (10.000,40.000) -- (10.000,40.000) -- cycle;
\draw [fill=c3] (7.500,7.500) -- (10.000,0.000) -- (8.000,8.000) -- (8.000,8.000) -- cycle;
\draw [fill=c5] (7.500,10.000) -- (8.000,8.000) -- (10.000,10.000) -- (10.000,10.000) -- cycle;
\draw [fill=c8] (7.500,15.000) -- (10.000,10.000) -- (8.000,16.000) -- (8.000,16.000) -- cycle;
\draw [fill=c10] (7.500,17.500) -- (8.000,16.000) -- (10.000,20.000) -- (10.000,20.000) -- cycle;
\draw [fill=c10] (7.500,22.500) -- (10.000,20.000) -- (8.000,24.000) -- (8.000,24.000) -- cycle;
\draw [fill=c8] (7.500,25.000) -- (8.000,24.000) -- (10.000,30.000) -- (10.000,30.000) -- cycle;
\draw [fill=c5] (7.500,30.000) -- (10.000,30.000) -- (8.000,32.000) -- (8.000,32.000) -- cycle;
\draw [fill=c3] (7.500,32.500) -- (8.000,32.000) -- (10.000,40.000) -- (10.000,40.000) -- cycle;
\draw [fill=c4] (8.000,8.000) -- (10.000,0.000) -- (10.000,10.000) -- (10.000,10.000) -- cycle;
\draw [fill=c9] (8.000,16.000) -- (10.000,10.000) -- (10.000,20.000) -- (10.000,20.000) -- cycle;
\draw [fill=c9] (8.000,24.000) -- (10.000,20.000) -- (10.000,30.000) -- (10.000,30.000) -- cycle;
\draw [fill=c4] (8.000,32.000) -- (10.000,30.000) -- (10.000,40.000) -- (10.000,40.000) -- cycle;
\draw [thick] (0.000,0.000) -- (10.000,0.000) -- (10.000,40.000) -- (0.000,40.000) -- (0.000,40.000) -- cycle;
\end{tikzpicture}
\end{center}
\caption{$BC(4,1)$ colored
using to the red and yellow palettes  (see  \textsection\ref{SubSecYR}).} \label{FigBC14}
\end{minipage}
% End of Lars's Fig 7
 \hspace{0.4cm}
 %%%%%%%%%%%%%%%%%%%%%%%%%%%%%%%%%%
 %%% Fig 8
 %%%%%%%%%%%%%%%%%%%%%%%%%%%%%%%%%%
   \begin{minipage}[b]{0.45\linewidth}
\centerline{\includegraphics[angle=0, width=1.25in]{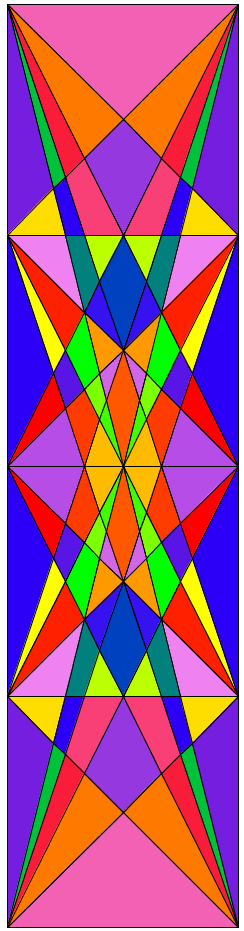}}
%\centerline{\includegraphics[angle=0, width=1.2in]{SC.4.1sSun.png}}
%\centerline{\includegraphics[angle=0, width=1.2in]{SC.4.1s2Gimp2.png}}
%\centerline{\includegraphics[angle=90, width=6.0in]{SC.4.1s2Gimp2.png}}
 % \centerline{\includegraphics[angle=90, width=6.0in]{tmp1.gif}}
   \caption{A version of $BC(4,1)$ colored by our `random coloring' algorithm
   (see  \textsection\ref{SubSecRan}).}
   \label{FigBC14c}
   \end{minipage}
\end{figure}

Of course we could equally well have started with a vertical rectangle of size $n \times 1$,
 in which case the graph would be denoted by $BC(n,1)$. Since this work was 
 partly inspired by the windows of  Gothic cathedrals, we admit to a slight preference 
 for $BC(n,1)$ over $BC(1,n)$, although as graphs they are isomorphic.
 Figs.~\ref{FigBC14} and \ref{FigBC14c} show our stained glass window $BC(4,1)$ using 
 two different coloring schemes.
 % Figs.~\ref{FigBC61}
 %and \ref{FigBC71} show $BC(6,1)$ and $BC(7,1)$.
 %The latter uses a different color scheme from the others in order to emphasize that there are only 
 %two kinds of cells in $BC(1,n)$ and $BC(n,1)$, namely triangles (shown in red in Fig.~\ref{FigBC61})
%and quadrilaterals (blue-green). We will return to this question  below.

 We will continue to discuss $BC(1,n)$, but the reader should remember that the results apply equally well to $BC(n,1)$.
  
% The next 2 figs are commented out to save time
%\begin{figure}[!ht]
%%%%%%%%%%%%%%%%%%%%%%%%%%%%%%%%%%%%%%
%%%%%%%%%%%%%%%%%%%         FIG 7  NOT
%%%%%%%%%%%%%%%%%%%%%%%%%%%%%%%%%%%%%%
%\begin{minipage}[b]{0.40\linewidth}
%\centerline{\includegraphics[angle=0, width=0.5\textwidth]{SC6.1.png}}
  %        \caption{$BC(6,1)$, based on a $6 \times 1$ array of squares.}
    %    \label{FigBC61}
     %   \end{minipage}
       %   \hspace{0.3cm}
%%%%%%%%%%%%%%%%%%%%%%%%%%%%%%%%%%%%%%
%%%%%%%%%%%%%%%%%%%         FIG 8  NOT
%%%%%%%%%%%%%%%%%%%%%%%%%%%%%%%%%%%%%%
 %         \begin{minipage}[b]{0.50\linewidth}
% \centerline{\includegraphics[angle=0, width=0.35\textwidth]{SC7.1.small.png}}
 %         \caption{$BC(7,1)$, showing  $360$ triangular cells (red) and $288$ quadrilaterals (blue-green).}
          %%%%%%% \caption{$BC(7,1)$, from a $7 \times 1$ array of squares. There are $360$ triangular cells (red) and $288$ quadrilaterals (blue-green).}
 %       \label{FigBC71}
  %     \end{minipage}
 %  \end{figure}

Another way to construct $BC(1,n)$ is to start with the complete bipartite graph
$K_{n+1,n+1}$ formed by taking $n+1$ equally spaced points in each of two horizontal rows,
joining every upper point to every lower point, and then adding the line segments through the two rows of points.
Thus $BC(1,2)$ in Fig.~\ref{FigBC12} is  the well-known nonplanar ``utilities'' graph $K_{3,3}$
if the two horizontal lines and the colors  are ignored.

The graphs $BC(1,n)$ %(or $BC(n,1)$) 
are one of the few families where there
are explicit formulas for the numbers of nodes ($\sN(1,n)$), edges ($\sE(1,n)$),
and cells ($\sC(1,n)$).  The initial values of these quantities are shown in Table~\ref{TabBC0},
along with the $A$-numbers in \cite{OEIS})
 of the corresponding sequences.

\begin{table}[!ht]
\caption{Numbers of nodes, edges, cells in $BC(1,n)$. }\label{TabBC0}
$$
   %\beql{EqBC0}
   \begin{array}{crrrrrrrrrrrc}
%\hline
   n: & 1 & 2 & 3 & 4 & 5 & 6 & 7 & 8 & 9 & 10 &  \cdots &  \cite{OEIS} \\
    \sN(1,n): & 5 &  13 &  35 &  75 &  159 &  275 &  477 &  755 &  1163 & 1659 & \cdots &  \seqnum{A331755}  \\
    \sE(1,n): & 8 &  28 &  80 &  178 &  372 &  654 &  1124 &  1782 &  2724 &  3914 & \cdots & \seqnum{A331757}  \\
    \sC(1,n): &4 &16 &46 &104 &214 &380 &648 &1028 &1562 &2256 &\cdots & \seqnum{A306302}
   %a(n): & 1 & 4 & 6 & 7 & 8 & 9 & 11 & 13 & 15 & 16 & 17 & 18 & \cdots
   \end{array}
 $$
 \end{table}
   
   Since by Euler's formula \eqn{Eq1},  $\sE(1,n) = \sN(1,n) + \sC(1,n) - 1$, there is no need to tabulate
   $\sE(1,n)$, and in future we shall omit those numbers. 
   
 The following  theorem is due to Legendre (2009) \cite{Leg09} and  Griffiths (2010) \cite{Gri10}, who 
 discuss the problem from the point of view of  $K_{n+1,n+1}$. 
 First we introduce an expression that will frequently appear in these formulas.
For $m,n,q \ge 1$, let
 \beql{EqVVV1}
      V(m,n,q) ~=~    \sum_{a=1..m} ~~   \sum_{\substack{b=1..n \\   \gcd\{a,b\}=q} }  (m+1-a)(n+1-b) \,.
 \eeq
   
 \begin{thm} (Legendre \cite[Prop.~6]{Leg09}, Griffiths \cite[Th.~3]{Gri10}.) \label{ThVRBC1}
For $n \ge 1$, the number of nodes in $BC(1,n)$, $\sN(1,n)$ (\seqnum{A331755}) is given by
\beql{EqVBC1}
\sN(1,n) ~=~ 2(n+1) + V(n,n,1) - V(n,n,2)\,, 
\eeq
and the number of cells, $\sC(1,n)$ (\seqnum{A306302}) is
 \beql{EqRBC1}
   \sC(1,n) ~=~   n^2 + 2n  + V(n,n,1)\,.
 \eeq
 \end{thm}
   
 \noindent Remarks: (i) A key step in the
 proof of~\eqn{EqVBC1} (see  \cite{Leg09}) is finding  a  condition for three chords to meet at a point. 
 (ii) The starting point for the proof of ~\eqn{EqRBC1} (see~\cite{Gri10})  is the observation
 that in the graph $BC(1,n)$ there are no interior edges that are parallel to the two long sides of the rectangle.
 This means that every cell has a unique node that is closest to the upper side of the rectangle.
 (iii) The term $2(n+1)$ on the right-hand side of \eqref{EqVBC1} is the number of nodes on the boundary of the rectangle. The difference between the other two terms is therefore the number of interior nodes in $BC(1,n)$  (\seqnum{A159065}):
 \beql{EqIBC1}
 1, 7, 27, 65, 147, 261, 461, 737, 1143, \cdots \,.
 \eeq
  
  %%%%%%%%%%%%%%%%%%%%%%%%%%%%%%%%%%%%%%
%%%%%%%%%%%%%%%%%%%         FIG 9
%%%%%%%%%%%%%%%%%%%%%%%%%%%%%%%%%%%%%%
%\begin{figure}[!ht]
  % \centerline{\includegraphics[angle=0, width=2.5in]{IT2cGimp.png}}
   %\caption{The isosceles triangle graph $IT(2)$. There are $14$ nodes ($7$ on boundary, $7$ in interior), $17$ cells
   %($15$ triangles, $2$ quadrilaterals), and
   %$30$ edges,  (Better black and white drawing needed)}
  % \label{FigIT2a}
  % \end{figure}

\begin{figure}[!ht]
\begin{center}
\begin{tikzpicture}
\draw (0,0) -- (0,6);
\draw (0,0) -- (6,0);
\draw (0,6) -- (6,0);
\draw (0,6) -- (3,0);
\draw (0,6) -- (2,0);
\draw (0,3) -- (6,0);
\draw (0,3) -- (3,0);
\draw (0,3) -- (2,0);
\draw (0,2) -- (6,0);
\draw (0,2) -- (3,0);
\draw (0,2) -- (2,0);
\end{tikzpicture}
\caption{The isosceles triangle graph $IT(2)$. There are $14$ nodes ($7$ on boundary, $7$ in interior), $17$ cells
   ($15$ triangles, $2$ quadrilaterals), and
   $30$ edges.}   \label{FigIT2a}
   \end{center}
\end{figure}
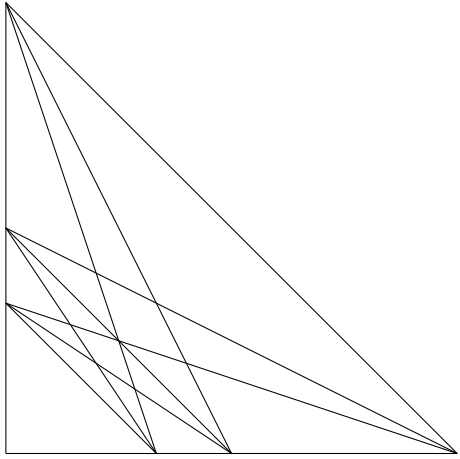

In 2019 Max Alekseyev added a comment to \seqnum{A306302} pointing out that the results in 
Theorem~\ref{ThVRBC1} are essentially the same as the results he and his coauthors had obtained
in \cite{ABZ15} (2015) 
for the isosceles triangle graphs $IT(n)$.

%%%%%%%%%%%%%%%%%%%%%%%%%%%%%%%%%%%%%%
% Section 3: %THE ISOSCELES TRIANGLE GRAPHS IT
%%%%%%%%%%%%%%%%%%%%%%%%%%%%%%%%%%%%%%

\section{The isosceles triangle graph \texorpdfstring{$IT(n)$}{IT(n)}.}\label{SecIT}

The definition of the \emph{isosceles triangle graph} $IT(n)$, $n\ge 1$, starts with an isosceles right triangle
with vertices $(0,0)$, $(0,1)$, and $(1,0)$. On the vertical side of the triangle  we place $n$ nodes at the points
$$
\left(0,\frac{1}{2}\right),
\left(0,\frac{1}{3}\right),
\left(0,\frac{1}{4}\right),
\ldots,
\left(0,\frac{1}{n+1}\right),
$$
and similarly on the horizontal side
we place $n$ nodes at the points
$$
\left(\frac{1}{2},0\right),
\left(\frac{1}{3},0\right),
\left(\frac{1}{4},0\right),
\ldots,
\left(\frac{1}{n+1},0\right)\,.
$$
There are no internal nodes on the hypotenuse.\footnote{In  Part 2 of this paper~\cite{Rose2}
we will discuss 
graphs formed by inserting $n$ equally-spaced nodes on all three sides of an equilateral triangle.}
We then draw chords between every pair of the $2n+3$ points on the boundary of the triangle.
Figs.~\ref{FigIT2a}, \ref{FigIT3}, \ref{FigIT4} show $IT(2)$, $IT(3)$ and $IT(4)$.
The latter two graphs have been colored using the red and yellow palettes (\textsection\ref{SubSecYR}).

\begin{figure}[!ht]
        %%%%%%%%%%%%%%%%%%%%%%%%%%%%%%%%%%%%%%
        %%%%%%%%%%%%%%%%%%%         FIG 10
        % Source in OEIS, entry A332358:
        %H Scott R. Shannon, <a href="/A332358/a332358.png">Illustration for n = 4</a>.

        %H Scott R. Shannon, <a href="/A332358/a332358_1.png">Illustration for n = 5</a>.
        %%%%%%%%%%%%%%%%%%%%%%%%%%%%%%%%%%%%%%
        \begin{minipage}[b]{0.45\linewidth}
                %%\centerline{\includegraphics[angle=0, width=0.95\textwidth]{IT_3.1..png}}

\begin{tikzpicture}[scale=0.069,black,line join=round]

\definecolor{c0}{RGB}{255,249,0}
\definecolor{c1}{RGB}{161,105,0}
\definecolor{c2}{RGB}{176,122,0}
\definecolor{c3}{RGB}{198,148,0}
\definecolor{c4}{RGB}{251,238,0}
\definecolor{c5}{RGB}{183,130,0}
\definecolor{c6}{RGB}{102,0,0}
\definecolor{c7}{RGB}{211,167,0}
\definecolor{c8}{RGB}{247,227,0}
\definecolor{c9}{RGB}{168,113,0}
\definecolor{c10}{RGB}{118,0,0}
\definecolor{c11}{RGB}{224,186,0}
\definecolor{c12}{RGB}{218,177,0}
\definecolor{c13}{RGB}{153,97,0}
\definecolor{c14}{RGB}{190,139,0}
\definecolor{c15}{RGB}{204,158,0}
\definecolor{c16}{RGB}{134,0,0}
\definecolor{c17}{RGB}{150,0,0}
\definecolor{c18}{RGB}{167,0,0}
\definecolor{c19}{RGB}{230,196,0}
\definecolor{c20}{RGB}{222,0,0}
\definecolor{c21}{RGB}{236,207,0}
\definecolor{c22}{RGB}{185,0,0}
\definecolor{c23}{RGB}{241,217,0}
\definecolor{c24}{RGB}{241,0,0}
\definecolor{c25}{RGB}{203,0,0}

\draw [fill=c0] (0.000,0.000) -- (25.000,0.000) -- (0.000,25.000) -- (0.000,25.000) -- cycle;
\draw [fill=c1] (0.000,25.000) -- (25.000,0.000) -- (14.286,14.286) -- (14.286,14.286) -- cycle;
\draw [fill=c2] (0.000,25.000) -- (14.286,14.286) -- (10.000,20.000) -- (10.000,20.000) -- cycle;
\draw [fill=c3] (0.000,25.000) -- (10.000,20.000) -- (7.692,23.077) -- (7.692,23.077) -- cycle;
\draw [fill=c4] (0.000,25.000) -- (7.692,23.077) -- (0.000,33.333) -- (0.000,33.333) -- cycle;
\draw [fill=c5] (0.000,33.333) -- (7.692,23.077) -- (11.111,22.222) -- (11.111,22.222) -- cycle;
\draw [fill=c6] (0.000,33.333) -- (11.111,22.222) -- (14.286,21.429) -- (12.500,25.000) -- (12.500,25.000) -- cycle;
\draw [fill=c7] (0.000,33.333) -- (12.500,25.000) -- (10.000,30.000) -- (10.000,30.000) -- cycle;
\draw [fill=c8] (0.000,33.333) -- (10.000,30.000) -- (0.000,50.000) -- (0.000,50.000) -- cycle;
\draw [fill=c9] (0.000,50.000) -- (10.000,30.000) -- (14.286,28.571) -- (14.286,28.571) -- cycle;
\draw [fill=c10] (0.000,50.000) -- (14.286,28.571) -- (18.182,27.273) -- (16.667,33.333) -- (16.667,33.333) -- cycle;
\draw [fill=c11] (0.000,50.000) -- (16.667,33.333) -- (14.286,42.857) -- (14.286,42.857) -- cycle;
\draw [fill=c12] (0.000,50.000) -- (14.286,42.857) -- (0.000,100.000) -- (0.000,100.000) -- cycle;
\draw [fill=c13] (0.000,100.000) -- (14.286,42.857) -- (20.000,40.000) -- (20.000,40.000) -- cycle;
\draw [fill=c14] (0.000,100.000) -- (20.000,40.000) -- (33.333,33.333) -- (33.333,33.333) -- cycle;
\draw [fill=c15] (0.000,100.000) -- (33.333,33.333) -- (100.000,0.000) -- (100.000,0.000) -- cycle;
%\draw [fill=black] (0.000,100.000) -- (100.000,0.000) -- (100.000,100.000) -- (100.000,100.000) -- cycle;
\draw [fill=c16] (7.692,23.077) -- (10.000,20.000) -- (16.667,16.667) -- (11.111,22.222) -- (11.111,22.222) -- cycle;
\draw [fill=c17] (10.000,20.000) -- (14.286,14.286) -- (20.000,10.000) -- (16.667,16.667) -- (16.667,16.667) -- cycle;
\draw [fill=c18] (10.000,30.000) -- (12.500,25.000) -- (20.000,20.000) -- (14.286,28.571) -- (14.286,28.571) -- cycle;
\draw [fill=c11] (11.111,22.222) -- (16.667,16.667) -- (14.286,21.429) -- (14.286,21.429) -- cycle;
\draw [fill=c19] (12.500,25.000) -- (14.286,21.429) -- (20.000,20.000) -- (20.000,20.000) -- cycle;
\draw [fill=c2] (14.286,14.286) -- (25.000,0.000) -- (20.000,10.000) -- (20.000,10.000) -- cycle;
\draw [fill=c20] (14.286,21.429) -- (16.667,16.667) -- (21.429,14.286) -- (20.000,20.000) -- (20.000,20.000) -- cycle;
\draw [fill=c21] (14.286,28.571) -- (20.000,20.000) -- (18.182,27.273) -- (18.182,27.273) -- cycle;
\draw [fill=c22] (14.286,42.857) -- (16.667,33.333) -- (25.000,25.000) -- (20.000,40.000) -- (20.000,40.000) -- cycle;
\draw [fill=c16] (16.667,16.667) -- (20.000,10.000) -- (23.077,7.692) -- (22.222,11.111) -- (22.222,11.111) -- cycle;
\draw [fill=c11] (16.667,16.667) -- (22.222,11.111) -- (21.429,14.286) -- (21.429,14.286) -- cycle;
\draw [fill=c23] (16.667,33.333) -- (18.182,27.273) -- (25.000,25.000) -- (25.000,25.000) -- cycle;
\draw [fill=c24] (18.182,27.273) -- (20.000,20.000) -- (27.273,18.182) -- (25.000,25.000) -- (25.000,25.000) -- cycle;
\draw [fill=c3] (20.000,10.000) -- (25.000,0.000) -- (23.077,7.692) -- (23.077,7.692) -- cycle;
\draw [fill=c19] (20.000,20.000) -- (21.429,14.286) -- (25.000,12.500) -- (25.000,12.500) -- cycle;
\draw [fill=c18] (20.000,20.000) -- (25.000,12.500) -- (30.000,10.000) -- (28.571,14.286) -- (28.571,14.286) -- cycle;
\draw [fill=c21] (20.000,20.000) -- (28.571,14.286) -- (27.273,18.182) -- (27.273,18.182) -- cycle;
\draw [fill=c25] (20.000,40.000) -- (25.000,25.000) -- (40.000,20.000) -- (33.333,33.333) -- (33.333,33.333) -- cycle;
\draw [fill=c6] (21.429,14.286) -- (22.222,11.111) -- (33.333,0.000) -- (25.000,12.500) -- (25.000,12.500) -- cycle;
\draw [fill=c5] (22.222,11.111) -- (23.077,7.692) -- (33.333,0.000) -- (33.333,0.000) -- cycle;
\draw [fill=c4] (23.077,7.692) -- (25.000,0.000) -- (33.333,0.000) -- (33.333,0.000) -- cycle;
\draw [fill=c7] (25.000,12.500) -- (33.333,0.000) -- (30.000,10.000) -- (30.000,10.000) -- cycle;
\draw [fill=c23] (25.000,25.000) -- (27.273,18.182) -- (33.333,16.667) -- (33.333,16.667) -- cycle;
\draw [fill=c22] (25.000,25.000) -- (33.333,16.667) -- (42.857,14.286) -- (40.000,20.000) -- (40.000,20.000) -- cycle;
\draw [fill=c10] (27.273,18.182) -- (28.571,14.286) -- (50.000,0.000) -- (33.333,16.667) -- (33.333,16.667) -- cycle;
\draw [fill=c9] (28.571,14.286) -- (30.000,10.000) -- (50.000,0.000) -- (50.000,0.000) -- cycle;
\draw [fill=c8] (30.000,10.000) -- (33.333,0.000) -- (50.000,0.000) -- (50.000,0.000) -- cycle;
\draw [fill=c11] (33.333,16.667) -- (50.000,0.000) -- (42.857,14.286) -- (42.857,14.286) -- cycle;
\draw [fill=c14] (33.333,33.333) -- (40.000,20.000) -- (100.000,0.000) -- (100.000,0.000) -- cycle;
\draw [fill=c13] (40.000,20.000) -- (42.857,14.286) -- (100.000,0.000) -- (100.000,0.000) -- cycle;
\draw [fill=c12] (42.857,14.286) -- (50.000,0.000) -- (100.000,0.000) -- (100.000,0.000) -- cycle;
\draw [thick] (0.000,0.000) -- (100.000,0.000)  -- (0.000,100.000) -- (0.000,100.000) -- cycle;
%\draw [thick] (0.000,0.000) -- (100.000,0.000) -- (100.000,100.000) -- (0.000,100.000) -- (0.000,100.000) -- cycle;
\end{tikzpicture}

                \caption{$IT(3)$ ($33$ triangles, $14$ quadrilaterals)}
                \label{FigIT3}
        \end{minipage}
        \hspace{0.3cm}
        %%%%%%%%%%%%%%%%%%%%%%%%%%%%%%%%%%%%%%
        %%%%%%%%%%%%%%%%%%%         FIG 11
        %%%%%%%%%%%%%%%%%%%%%%%%%%%%%%%%%%%%%%
        \begin{minipage}[b]{0.45\linewidth}
                %%\centerline{\includegraphics[angle=0, width=0.95\textwidth]{IT_4s.png}}
\begin{tikzpicture}[scale=0.069,black,line join=round]

\definecolor{c0}{RGB}{255,255,0}
\definecolor{c1}{RGB}{157,101,0}
\definecolor{c2}{RGB}{167,112,0}
\definecolor{c3}{RGB}{184,132,0}
\definecolor{c4}{RGB}{204,157,0}
\definecolor{c5}{RGB}{255,250,0}
\definecolor{c6}{RGB}{174,120,0}
\definecolor{c7}{RGB}{181,128,0}
\definecolor{c8}{RGB}{194,144,0}
\definecolor{c9}{RGB}{213,170,0}
\definecolor{c10}{RGB}{254,245,0}
\definecolor{c11}{RGB}{171,116,0}
\definecolor{c12}{RGB}{188,136,0}
\definecolor{c13}{RGB}{201,152,0}
\definecolor{c14}{RGB}{228,192,0}
\definecolor{c15}{RGB}{252,240,0}
\definecolor{c16}{RGB}{164,108,0}
\definecolor{c17}{RGB}{177,124,0}
\definecolor{c18}{RGB}{197,148,0}
\definecolor{c19}{RGB}{243,220,0}
\definecolor{c20}{RGB}{216,174,0}
\definecolor{c21}{RGB}{153,97,0}
\definecolor{c22}{RGB}{160,104,0}
\definecolor{c23}{RGB}{191,140,0}
\definecolor{c24}{RGB}{219,178,0}
\definecolor{c25}{RGB}{116,0,0}
\definecolor{c26}{RGB}{102,0,0}
\definecolor{c27}{RGB}{123,0,0}
\definecolor{c28}{RGB}{145,0,0}
\definecolor{c29}{RGB}{168,0,0}
\definecolor{c30}{RGB}{109,0,0}
\definecolor{c31}{RGB}{209,0,0}
\definecolor{c32}{RGB}{153,0,0}
\definecolor{c33}{RGB}{207,161,0}
\definecolor{c34}{RGB}{210,165,0}
\definecolor{c35}{RGB}{222,183,0}
\definecolor{c36}{RGB}{192,0,0}
\definecolor{c37}{RGB}{161,0,0}
\definecolor{c38}{RGB}{217,0,0}
\definecolor{c39}{RGB}{176,0,0}
\definecolor{c40}{RGB}{225,187,0}
\definecolor{c41}{RGB}{241,216,0}
\definecolor{c42}{RGB}{138,0,0}
\definecolor{c43}{RGB}{230,197,0}
\definecolor{c44}{RGB}{233,201,0}
\definecolor{c45}{RGB}{184,0,0}
\definecolor{c46}{RGB}{200,0,0}
\definecolor{c47}{RGB}{234,0,0}
\definecolor{c48}{RGB}{238,211,0}
\definecolor{c49}{RGB}{236,206,0}
\definecolor{c50}{RGB}{245,225,0}
\definecolor{c51}{RGB}{130,0,0}
\definecolor{c52}{RGB}{248,230,0}
\definecolor{c53}{RGB}{251,0,0}
\definecolor{c54}{RGB}{250,235,0}
\definecolor{c55}{RGB}{243,0,0}
\definecolor{c56}{RGB}{226,0,0}

\draw [fill=c0] (0.000,0.000) -- (20.000,0.000) -- (0.000,20.000) -- (0.000,20.000) -- cycle;
\draw [fill=c1] (0.000,20.000) -- (20.000,0.000) -- (11.111,11.111) -- (11.111,11.111) -- cycle;
\draw [fill=c2] (0.000,20.000) -- (11.111,11.111) -- (7.692,15.385) -- (7.692,15.385) -- cycle;
\draw [fill=c3] (0.000,20.000) -- (7.692,15.385) -- (5.882,17.647) -- (5.882,17.647) -- cycle;
\draw [fill=c4] (0.000,20.000) -- (5.882,17.647) -- (4.762,19.048) -- (4.762,19.048) -- cycle;
\draw [fill=c5] (0.000,20.000) -- (4.762,19.048) -- (0.000,25.000) -- (0.000,25.000) -- cycle;
\draw [fill=c6] (0.000,25.000) -- (4.762,19.048) -- (6.250,18.750) -- (6.250,18.750) -- cycle;
\draw [fill=c7] (0.000,25.000) -- (6.250,18.750) -- (9.091,18.182) -- (9.091,18.182) -- cycle;
\draw [fill=c8] (0.000,25.000) -- (9.091,18.182) -- (7.143,21.429) -- (7.143,21.429) -- cycle;
\draw [fill=c9] (0.000,25.000) -- (7.143,21.429) -- (5.882,23.529) -- (5.882,23.529) -- cycle;
\draw [fill=c10] (0.000,25.000) -- (5.882,23.529) -- (0.000,33.333) -- (0.000,33.333) -- cycle;
\draw [fill=c11] (0.000,33.333) -- (5.882,23.529) -- (7.692,23.077) -- (7.692,23.077) -- cycle;
\draw [fill=c12] (0.000,33.333) -- (7.692,23.077) -- (11.111,22.222) -- (11.111,22.222) -- cycle;
\draw [fill=c13] (0.000,33.333) -- (11.111,22.222) -- (9.091,27.273) -- (9.091,27.273) -- cycle;
\draw [fill=c14] (0.000,33.333) -- (9.091,27.273) -- (7.692,30.769) -- (7.692,30.769) -- cycle;
\draw [fill=c15] (0.000,33.333) -- (7.692,30.769) -- (0.000,50.000) -- (0.000,50.000) -- cycle;
\draw [fill=c16] (0.000,50.000) -- (7.692,30.769) -- (10.000,30.000) -- (10.000,30.000) -- cycle;
\draw [fill=c17] (0.000,50.000) -- (10.000,30.000) -- (14.286,28.571) -- (14.286,28.571) -- cycle;
\draw [fill=c18] (0.000,50.000) -- (14.286,28.571) -- (12.500,37.500) -- (12.500,37.500) -- cycle;
\draw [fill=c19] (0.000,50.000) -- (12.500,37.500) -- (11.111,44.444) -- (11.111,44.444) -- cycle;
\draw [fill=c20] (0.000,50.000) -- (11.111,44.444) -- (0.000,100.000) -- (0.000,100.000) -- cycle;
\draw [fill=c21] (0.000,100.000) -- (11.111,44.444) -- (14.286,42.857) -- (14.286,42.857) -- cycle;
\draw [fill=c22] (0.000,100.000) -- (14.286,42.857) -- (20.000,40.000) -- (20.000,40.000) -- cycle;
\draw [fill=c23] (0.000,100.000) -- (20.000,40.000) -- (33.333,33.333) -- (33.333,33.333) -- cycle;
\draw [fill=c24] (0.000,100.000) -- (33.333,33.333) -- (100.000,0.000) -- (100.000,0.000) -- cycle;
%\draw [fill=black] (0.000,100.000) -- (100.000,0.000) -- (100.000,100.000) -- (100.000,100.000) -- cycle;
\draw [fill=c25] (4.762,19.048) -- (5.882,17.647) -- (8.333,16.667) -- (6.250,18.750) -- (6.250,18.750) -- cycle;
\draw [fill=c26] (5.882,17.647) -- (7.692,15.385) -- (12.500,12.500) -- (8.333,16.667) -- (8.333,16.667) -- cycle;
\draw [fill=c27] (5.882,23.529) -- (7.143,21.429) -- (10.000,20.000) -- (7.692,23.077) -- (7.692,23.077) -- cycle;
\draw [fill=c28] (6.250,18.750) -- (8.333,16.667) -- (10.526,15.789) -- (9.091,18.182) -- (9.091,18.182) -- cycle;
\draw [fill=c29] (7.143,21.429) -- (9.091,18.182) -- (11.765,17.647) -- (10.000,20.000) -- (10.000,20.000) -- cycle;
\draw [fill=c30] (7.692,15.385) -- (11.111,11.111) -- (15.385,7.692) -- (12.500,12.500) -- (12.500,12.500) -- cycle;
\draw [fill=c31] (7.692,23.077) -- (10.000,20.000) -- (12.500,18.750) -- (11.111,22.222) -- (11.111,22.222) -- cycle;
\draw [fill=c32] (7.692,30.769) -- (9.091,27.273) -- (12.500,25.000) -- (10.000,30.000) -- (10.000,30.000) -- cycle;
\draw [fill=c33] (8.333,16.667) -- (12.500,12.500) -- (10.526,15.789) -- (10.526,15.789) -- cycle;
\draw [fill=c34] (9.091,18.182) -- (10.526,15.789) -- (14.286,14.286) -- (14.286,14.286) -- cycle;
\draw [fill=c35] (9.091,18.182) -- (14.286,14.286) -- (11.765,17.647) -- (11.765,17.647) -- cycle;
\draw [fill=c36] (9.091,27.273) -- (11.111,22.222) -- (14.286,21.429) -- (12.500,25.000) -- (12.500,25.000) -- cycle;
\draw [fill=c37] (10.000,20.000) -- (11.765,17.647) -- (13.043,17.391) -- (12.500,18.750) -- (12.500,18.750) -- cycle;
\draw [fill=c38] (10.000,30.000) -- (12.500,25.000) -- (15.385,23.077) -- (14.286,28.571) -- (14.286,28.571) -- cycle;
\draw [fill=c39] (10.526,15.789) -- (12.500,12.500) -- (15.789,10.526) -- (14.286,14.286) -- (14.286,14.286) -- cycle;
\draw [fill=c2] (11.111,11.111) -- (20.000,0.000) -- (15.385,7.692) -- (15.385,7.692) -- cycle;
\draw [fill=c40] (11.111,22.222) -- (12.500,18.750) -- (16.667,16.667) -- (16.667,16.667) -- cycle;
\draw [fill=c41] (11.111,22.222) -- (16.667,16.667) -- (14.286,21.429) -- (14.286,21.429) -- cycle;
\draw [fill=c42] (11.111,44.444) -- (12.500,37.500) -- (16.667,33.333) -- (14.286,42.857) -- (14.286,42.857) -- cycle;
\draw [fill=c43] (11.765,17.647) -- (14.286,14.286) -- (13.043,17.391) -- (13.043,17.391) -- cycle;
\draw [fill=c26] (12.500,12.500) -- (15.385,7.692) -- (17.647,5.882) -- (16.667,8.333) -- (16.667,8.333) -- cycle;
\draw [fill=c33] (12.500,12.500) -- (16.667,8.333) -- (15.789,10.526) -- (15.789,10.526) -- cycle;
\draw [fill=c44] (12.500,18.750) -- (13.043,17.391) -- (16.667,16.667) -- (16.667,16.667) -- cycle;
\draw [fill=c45] (12.500,25.000) -- (14.286,21.429) -- (15.789,21.053) -- (15.385,23.077) -- (15.385,23.077) -- cycle;
\draw [fill=c46] (12.500,37.500) -- (14.286,28.571) -- (18.182,27.273) -- (16.667,33.333) -- (16.667,33.333) -- cycle;
\draw [fill=c47] (13.043,17.391) -- (14.286,14.286) -- (17.391,13.043) -- (16.667,16.667) -- (16.667,16.667) -- cycle;
\draw [fill=c34] (14.286,14.286) -- (15.789,10.526) -- (18.182,9.091) -- (18.182,9.091) -- cycle;
\draw [fill=c35] (14.286,14.286) -- (18.182,9.091) -- (17.647,11.765) -- (17.647,11.765) -- cycle;
\draw [fill=c43] (14.286,14.286) -- (17.647,11.765) -- (17.391,13.043) -- (17.391,13.043) -- cycle;
\draw [fill=c48] (14.286,21.429) -- (16.667,16.667) -- (15.789,21.053) -- (15.789,21.053) -- cycle;
\draw [fill=c49] (14.286,28.571) -- (15.385,23.077) -- (20.000,20.000) -- (20.000,20.000) -- cycle;
\draw [fill=c50] (14.286,28.571) -- (20.000,20.000) -- (18.182,27.273) -- (18.182,27.273) -- cycle;
\draw [fill=c51] (14.286,42.857) -- (16.667,33.333) -- (25.000,25.000) -- (20.000,40.000) -- (20.000,40.000) -- cycle;
\draw [fill=c3] (15.385,7.692) -- (20.000,0.000) -- (17.647,5.882) -- (17.647,5.882) -- cycle;
\draw [fill=c52] (15.385,23.077) -- (15.789,21.053) -- (20.000,20.000) -- (20.000,20.000) -- cycle;
\draw [fill=c28] (15.789,10.526) -- (16.667,8.333) -- (18.750,6.250) -- (18.182,9.091) -- (18.182,9.091) -- cycle;
\draw [fill=c53] (15.789,21.053) -- (16.667,16.667) -- (21.053,15.789) -- (20.000,20.000) -- (20.000,20.000) -- cycle;
\draw [fill=c25] (16.667,8.333) -- (17.647,5.882) -- (19.048,4.762) -- (18.750,6.250) -- (18.750,6.250) -- cycle;
\draw [fill=c44] (16.667,16.667) -- (17.391,13.043) -- (18.750,12.500) -- (18.750,12.500) -- cycle;
\draw [fill=c40] (16.667,16.667) -- (18.750,12.500) -- (22.222,11.111) -- (22.222,11.111) -- cycle;
\draw [fill=c41] (16.667,16.667) -- (22.222,11.111) -- (21.429,14.286) -- (21.429,14.286) -- cycle;
\draw [fill=c48] (16.667,16.667) -- (21.429,14.286) -- (21.053,15.789) -- (21.053,15.789) -- cycle;
\draw [fill=c54] (16.667,33.333) -- (18.182,27.273) -- (25.000,25.000) -- (25.000,25.000) -- cycle;
\draw [fill=c37] (17.391,13.043) -- (17.647,11.765) -- (20.000,10.000) -- (18.750,12.500) -- (18.750,12.500) -- cycle;
\draw [fill=c4] (17.647,5.882) -- (20.000,0.000) -- (19.048,4.762) -- (19.048,4.762) -- cycle;
\draw [fill=c29] (17.647,11.765) -- (18.182,9.091) -- (21.429,7.143) -- (20.000,10.000) -- (20.000,10.000) -- cycle;
\draw [fill=c7] (18.182,9.091) -- (18.750,6.250) -- (25.000,0.000) -- (25.000,0.000) -- cycle;
\draw [fill=c8] (18.182,9.091) -- (25.000,0.000) -- (21.429,7.143) -- (21.429,7.143) -- cycle;
\draw [fill=c55] (18.182,27.273) -- (20.000,20.000) -- (27.273,18.182) -- (25.000,25.000) -- (25.000,25.000) -- cycle;
\draw [fill=c6] (18.750,6.250) -- (19.048,4.762) -- (25.000,0.000) -- (25.000,0.000) -- cycle;
\draw [fill=c31] (18.750,12.500) -- (20.000,10.000) -- (23.077,7.692) -- (22.222,11.111) -- (22.222,11.111) -- cycle;
\draw [fill=c5] (19.048,4.762) -- (20.000,0.000) -- (25.000,0.000) -- (25.000,0.000) -- cycle;
\draw [fill=c27] (20.000,10.000) -- (21.429,7.143) -- (23.529,5.882) -- (23.077,7.692) -- (23.077,7.692) -- cycle;
\draw [fill=c52] (20.000,20.000) -- (21.053,15.789) -- (23.077,15.385) -- (23.077,15.385) -- cycle;
\draw [fill=c49] (20.000,20.000) -- (23.077,15.385) -- (28.571,14.286) -- (28.571,14.286) -- cycle;
\draw [fill=c50] (20.000,20.000) -- (28.571,14.286) -- (27.273,18.182) -- (27.273,18.182) -- cycle;
\draw [fill=c56] (20.000,40.000) -- (25.000,25.000) -- (40.000,20.000) -- (33.333,33.333) -- (33.333,33.333) -- cycle;
\draw [fill=c45] (21.053,15.789) -- (21.429,14.286) -- (25.000,12.500) -- (23.077,15.385) -- (23.077,15.385) -- cycle;
\draw [fill=c9] (21.429,7.143) -- (25.000,0.000) -- (23.529,5.882) -- (23.529,5.882) -- cycle;
\draw [fill=c36] (21.429,14.286) -- (22.222,11.111) -- (27.273,9.091) -- (25.000,12.500) -- (25.000,12.500) -- cycle;
\draw [fill=c12] (22.222,11.111) -- (23.077,7.692) -- (33.333,0.000) -- (33.333,0.000) -- cycle;
\draw [fill=c13] (22.222,11.111) -- (33.333,0.000) -- (27.273,9.091) -- (27.273,9.091) -- cycle;
\draw [fill=c11] (23.077,7.692) -- (23.529,5.882) -- (33.333,0.000) -- (33.333,0.000) -- cycle;
\draw [fill=c38] (23.077,15.385) -- (25.000,12.500) -- (30.000,10.000) -- (28.571,14.286) -- (28.571,14.286) -- cycle;
\draw [fill=c10] (23.529,5.882) -- (25.000,0.000) -- (33.333,0.000) -- (33.333,0.000) -- cycle;
\draw [fill=c32] (25.000,12.500) -- (27.273,9.091) -- (30.769,7.692) -- (30.000,10.000) -- (30.000,10.000) -- cycle;
\draw [fill=c54] (25.000,25.000) -- (27.273,18.182) -- (33.333,16.667) -- (33.333,16.667) -- cycle;
\draw [fill=c51] (25.000,25.000) -- (33.333,16.667) -- (42.857,14.286) -- (40.000,20.000) -- (40.000,20.000) -- cycle;
\draw [fill=c14] (27.273,9.091) -- (33.333,0.000) -- (30.769,7.692) -- (30.769,7.692) -- cycle;
\draw [fill=c46] (27.273,18.182) -- (28.571,14.286) -- (37.500,12.500) -- (33.333,16.667) -- (33.333,16.667) -- cycle;
\draw [fill=c17] (28.571,14.286) -- (30.000,10.000) -- (50.000,0.000) -- (50.000,0.000) -- cycle;
\draw [fill=c18] (28.571,14.286) -- (50.000,0.000) -- (37.500,12.500) -- (37.500,12.500) -- cycle;
\draw [fill=c16] (30.000,10.000) -- (30.769,7.692) -- (50.000,0.000) -- (50.000,0.000) -- cycle;
\draw [fill=c15] (30.769,7.692) -- (33.333,0.000) -- (50.000,0.000) -- (50.000,0.000) -- cycle;
\draw [fill=c42] (33.333,16.667) -- (37.500,12.500) -- (44.444,11.111) -- (42.857,14.286) -- (42.857,14.286) -- cycle;
\draw [fill=c23] (33.333,33.333) -- (40.000,20.000) -- (100.000,0.000) -- (100.000,0.000) -- cycle;
\draw [fill=c19] (37.500,12.500) -- (50.000,0.000) -- (44.444,11.111) -- (44.444,11.111) -- cycle;
\draw [fill=c22] (40.000,20.000) -- (42.857,14.286) -- (100.000,0.000) -- (100.000,0.000) -- cycle;
\draw [fill=c21] (42.857,14.286) -- (44.444,11.111) -- (100.000,0.000) -- (100.000,0.000) -- cycle;
\draw [fill=c20] (44.444,11.111) -- (50.000,0.000) -- (100.000,0.000) -- (100.000,0.000) -- cycle;
\draw [thick] (0.000,0.000) -- (100.000,0.000) -- (0.000,100.000) -- (0.000,100.000) -- cycle;
\end{tikzpicture}

                \caption{$IT(4)$ ($71$ triangles, $34$ quadrilaterals)}
                \label{FigIT4}
        \end{minipage}
\end{figure}

%%%%%%%%%%%%%%%%%%%%%%%%%%%%%%%%%%%%%%
%%%%%%%%%%%%%%%%%%%         FIG 12
%%%%%%%%%%%%%%%%%%%%%%%%%%%%%%%%%%%%%%

% The new Fig 12
% My drawing of Max's map
\begin{figure}[!ht]
\begin{tikzpicture}

\draw (0,0) -- (7,0);
\draw (0,0) -- (6,3);
\draw (0,0) -- (3,3);
\draw (0,0) -- (0,3);
\draw (0,3) -- (3,0);
\draw (0,3) -- (7,3);
\draw (0,3) -- (6,0);
\draw (3,0) -- (3,3);
\draw (3,0) -- (6,3);
\draw (3,3) -- (6,0);
\draw (6,0) -- (6,3);

\draw (9+0,0) -- (9+0,6);
\draw (9+0,0) -- (9+6,0);
\draw (9+0,6) -- (9+6,0);
\draw (9+0,6) -- (9+3,0);
\draw (9+0,6) -- (9+2,0);
\draw (9+0,3) -- (9+6,0);
\draw (9+0,3) -- (9+3,0);
\draw (9+0,3) -- (9+2,0);
\draw (9+0,2) -- (9+6,0);
\draw (9+0,2) -- (9+3,0);
\draw (9+0,2) -- (9+2,0);

\draw[->][ultra thick] (7,1.5) -- (8,1.5);

\node[below] at (0,0) {A};
\node[below] at (3,0) {B};
\node[below] at (6,0) {C};
\node[below] at (7,0) {D};
\node[above] at (0,3) {E};
\node[above] at (3,3) {F};
\node[above] at (6,3) {G};
\node[above] at (7,3) {D};

\node[left] at (9,6) {E};
\node[left] at (9,3) {F};
\node[left] at (9,2) {G};
\node[below] at (8.8,0) {D};
\node[below] at (11,0) {C};
\node[below] at (12,0) {B};
\node[below] at (15,0) {A};

\end{tikzpicture}
\caption{Illustrating the map \eqn{EqMaxMap} from $BC(1,2)$ to $IT(2)$.}
\label{FigMaxMap2}
\end{figure}
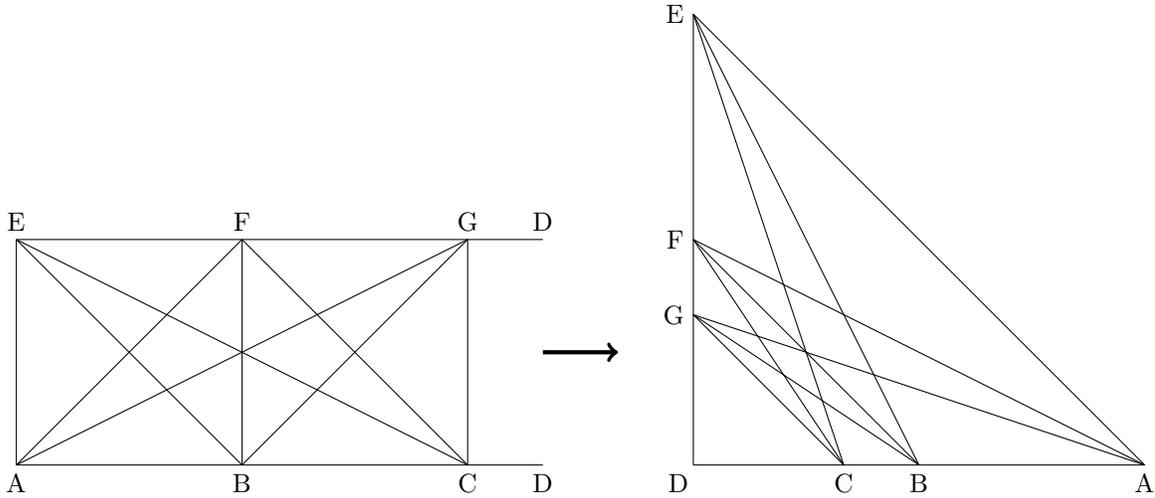

% the old text for Fig 12
%\begin{figure}[!ht]
 %  \centerline{\includegraphics[angle=0, width=5.5in]{MaxMap3Gimp.png}}
  % \caption{Illustrating the map \eqn{EqMaxMap} from $BC(1,2)$ to $IT(2)$.}
  % \label{FigMaxMap2}
   %\end{figure}

%%%%%%%%%%%%%%%%%%%%%%%%%%%%%%%%%%%%%%
%%%%%%%%%%%%%%%%%%%         FIG 12
% The original drawing is in A324042, at the link
% %H Jinyuan Wang, <a href="/A324042/a324042.png">Illustration for n = 1, 2, 3, 4, 5</a>
% -rw-r--r--  1 njasloane  staff    59142 Jul 24 21:06 Wang_BC1nGimp.png
% -rw-r--r--  1 njasloane  staff    75495 Jul 24 21:06 Wang_BC1n.xcf
% -rw-r--r--@ 1 njasloane  staff  1863505 Jul 23 14:55 Wang_BC1n.png
%%%%%%%%%%%%%%%%%%%%%%%%%%%%%%%%%%%%%%
%\begin{figure}[!ht]
  % \centerline{\includegraphics[angle=0, width=5.5in]{Wang_BC1nGimp.png}}
 %  \caption{$BC(1,n)$ for $n=1,\dots,5$, with
% the triangular regions highlighted for $n=1,2,3$. Figure courtesy of Jinyuan Wang.}
 %  \label{FigBC1nWang}
 %  \end{figure}

Alekseyev pointed out that if we take the boundary points of $BC(1,n)$ to be the
points $(i,0)$ and $(i,1)$ for $i = 0,\ldots,n$, then the map
\beql{EqMaxMap}
(x,y) ~\longmapsto~ \left( \frac{1-y}{x+1}, \frac{y}{x+1} \right)\,,
\eeq
maps $BC(1,n)$ onto $IT(n)$ minus the node and cell at the origin.
Figure~\ref{FigMaxMap2} illustrates this in the case $n=2$.
The six boundary nodes $A, B, C, E, F, G$ of $BC(1,2)$
are mapped to six of the seven boundary nodes of $IT(2)$.
The point $D$,  the point at infinity on the positive $x$ axis 
(not part of $BC(1,2)$), is mapped to the origin in $IT(2)$.
The  region $D,C,G,D$ to the right of $BC(1,2)$ is mapped to the triangular cell $D, G, C, D$ at the origin 
in $IT(2)$.

A similar thing happens in the general case:  $IT(n)$
always has one more node than $BC(1,n)$, two more edges, and one more cell.
When these adjustments are made to the formulas in Theorem~\ref{ThVRBC1},
we obtain the formulas in Theorem~13 of \cite{ABZ15}.
The counts for nodes, edges, and cells in $IT(n)$ are given in
\seqnum{A332632}, \seqnum{A332360}, and \seqnum{A332358}.

However, Alekseyev (personal communication) also pointed out that Theorem~13 of \cite{ABZ15} mentions
an additional property of $IT(n)$---and hence of $BC(1,n)$---that seems to have been overlooked in \cite{Leg09} and \cite{Gri10}:

\begin{thm} (Alekseyev et al.~\cite{ABZ15}):\label{ThABZ}
The cells in $IT(n)$ and hence $BC(1,n)$ are either triangles or quadrilaterals.
\end{thm}

That is, no cell  in $BC(1,n)$ has five or more edges. The proof in \cite{ABZ15} depends on a theorem
 about teaching sets for threshold function~\cite{SZ98, Zol01}. No other proof 
seems to be known.  We therefore state:

\begin{problem}\label{OP1}
Find a purely geometrical proof of Theorem~\ref{ThABZ}.
\end{problem}

%%%%%%%%%%%%%%%%%%%%%%%%%%%%%%%%%%%%%%%%%%%
%%%   Section 4  THE CELLS IN BC(1,n)
%%%%%%%%%%%%%%%%%%%%%%%%%%%%%%%%%%%%%%%%%%%

\section{The cells in \texorpdfstring{$BC(1,n)$}{BC(1,n)}}\label{SecBCcells}

From Theorems~\ref{ThVRBC1} and \ref{ThABZ}, we can determine the numbers of
triangular  and quadrilateral cells in $BC(1,n)$
(sequences  \seqnum{A324042}  and  \seqnum{A324043}).

\begin{thm}\label{ThTQBC}
The $\sC(1,n)$ cells in $BC(1,n)$ are made up of
\beql{EqTBC1}
T(n) = 2V(n,n,2) + 2n(n+1)
\eeq
triangles and
\beql{EqQBC1}
Q(n) = V(n,n,1)-2V(n,n,2) -n^2
\eeq
quadrilaterals.
\end{thm}
\begin{proof}
%Let $T$, $Q$ denote the numbers of triangular and quadrilateral cells in $BC(1,n)$. Then
%\beql{EqTQ1}
%T+Q=\sC(n)\,.
%\eeq
The sum $3T(n)+4Q(n)$ double-counts the edges in $BC(1,n)$ except that the $2n+2$ boundary edges are counted only once. Therefore
\beql{EqTQ2}
3T(n)+4Q(n)+(2n+2) = 2\sE(1,n) = 2(\sN(1,n) + \sC(1,n) -1)\,,
\eeq
and of course by Theorem~\ref{ThABZ}, $T(n)+Q(n)=\sC(1,n)$.
The proof is completed by solving these two equations  for $T(n)$ and $Q(n)$ and using \eqn{EqVBC1}, \eqn{EqRBC1}.
\end{proof}

Figures~\ref{FigBC11bw}, \ref{FigBC12}, \ref{FigBC13}, and \ref{FigBC14} show 
the triangles and quadrilaterals for  $n=1,...,4$.

%Fig.~\ref{FigBC1nWang} shows $BC(1,n)$ for $n=1,\dots,5$, with
%the triangular regions highlighted for $n=1,2,3$.
%  and Fig.~\ref{FigBC71} shows the $360$ triangular and $288$ quadrilateral cells in $BC(7,1)$.

One way to attack Open Problem~\ref{OP1} is to try to understand the distribution of
cells in each of the $n$ squares of $BC(1,n)$. Let $t_{n,k}$, $q_{n,k}$, and $c_{n,k}$ 
denote the numbers of triangles, quadrilaterals, and cells in the $k$-th square of
$BC(1,n)$ for $1 \le k \le n$ (so $t_{n,k}+q_{n,k}=c_{n,k}$ and
$\sum_k c_{n,k} = \sC(1,n)$). From Fig.~\ref{FigBC12}, for example, we see that $t_{1,1}=t_{1,2}=7$,
$q_{1,1}=q_{1,2}=1$, and $c_{1,1}=c_{1,2}=8$.

The two end squares of $BC(1,n)$ are easily understood, and for future reference we state the result as:

\begin{thm}\label{ThBCcnr}
For $n \ge 2$, the two end squares of $BC(1,n)$ both contain $2n+3$ triangles and $2n-3$
quadrilaterals.
\end{thm}.

\begin{table}[!ht] % A333286
\caption{Number $t_{n,k}$ of triangles  in $k$-th square in $BC(1,n)$ (\seqnum{A333286}). }\label{TabWangt}
$$
   \begin{array}{c|rrrrrrrrrr}
n \backslash k & 1 & 2 & 3 & 4 & 5 & 6 & 7 & 8 & 9 & 10   \\
\hline
1 &  4 &  & & & & & & & & \\
2 &  7 &  7 &  & & & & & & &  \\
3 &  9 &  14 &  9 &  & & & & & &  \\
4 &  11 &  24 &  24 &  11 &  & & & & &  \\
5 &  13 &  30 &  38 &  30 &  13 &  & & & & \\
6 &  15 &  38 &  60 &  60 &  38 &  15 &  & & & \\
7 &  17 &  44 &  76 &  86 &  76 &  44 &  17   & & & \\
8 &  19 &  52 &  92 &  120 &  120 &  92 &  52 &  19   & & \\
9 &  21 &  58 &  106 &  146 &  158 &  146 &  106 &  58 &  21 &   \\
10 &  23 &  66 &  126 &  178 &  216 &  216 &  178 &  126 &  66 &  23   
   \end{array}
 $$
 \end{table}

\begin{table}[!ht] % quadrilaterals A333287
\caption{Number $q_{n,k}$ of quadrilaterals  in $k$-th square  in $BC(1,n)$ (\seqnum{A333287}). }\label{TabWangq}
$$
   \begin{array}{c|rrrrrrrrrr}
n \backslash k & 1 & 2 & 3 & 4 & 5 & 6 & 7 & 8 & 9 & 10   \\
\hline
1 &  0 & &&&&&&&& \\
2 &  1 &  1 & &&&&&&& \\
3 &  3 &  8 &  3 & &&&&&& \\
4 &  5 &  12 &  12 &  5 & &&&&& \\
5 &  7 &  22 &  32 &  22 &  7 &  &&&& \\
6 &  9 &  28 &  40 &  40 &  28 &  9 &  &&& \\
7 &  11 &  38 &  58 &  74 &  58 &  38 &  11 &  & & \\
8 &  13 &  46 &  74 &  98 &  98 &  74 &  46 &  13 &  & \\
9 &  15 &  58 &  92 &  130 &  152 &  130 &  92 &  58 &  15 &  \\
10 &  17 &  68 &  104 &  150 &  180 &  180 &  150 &  104 &  68 &  17 \\
   \end{array}
 $$
 \end{table}

\begin{table}[!ht] % cells A333288
\caption{Total number $c_{n,k}$ of cells  in $k$-th square  in $BC(1,n)$ (\seqnum{A333288}). }\label{TabWangc}
$$
   \begin{array}{c|rrrrrrrrrr}
n \backslash k & 1 & 2 & 3 & 4 & 5 & 6 & 7 & 8 & 9 & 10   \\
\hline
1 &  4 &  &&&&&&&& \\
2 &  8 &  8 &  &&&&&&& \\
3 &  12 &  22 &  12 &  &&&&&& \\
4 &  16 &  36 &  36 &  16 &  &&&&& \\
5 &  20 &  52 &  70 &  52 &  20 &  &&&& \\
6 &  24 &  66 &  100 &  100 &  66 &  24 &  &&& \\
7 &  28 &  82 &  134 &  160 &  134 &  82 &  28 &  && \\
8 &  32 &  98 &  166 &  218 &  218 &  166 &  98 &  32 & & \\
9 &  36 &  116 &  198 &  276 &  310 &  276 &  198 &  116 &  36 & \\
10 &  40 &  134 &  230 &  328 &  396 &  396 &  328 &  230 &  134 &  40 
   \end{array}
 $$
 \end{table}

Tables~\ref{TabWangt}, \ref{TabWangq}, and \ref{TabWangc} show the values of $t_{n,k}$,
$q_{n,k}$, and $c_{n,k}$ for $n \le 10$.  More extensive tables,  for $n \le 80$, are given
in entries \seqnum{A333286}, \seqnum{A333287}, \seqnum{A333288}.  
However, even with $80$ rows of data, we have been unable to find formulas for these numbers.

There is certainly a lot of structure in these tables.
Using the Salvy-Zimmermann \emph{gfun} Maple program \cite{GFUN}, 
we attempted  to find generating functions
for the columns of these tables.  On the basis of admittedly little evidence, 
we make the following conjecture.

\begin{conj}\label{BC1conj1}
In all three of Tables~\ref{TabWangt}, \ref{TabWangq}, and \ref{TabWangc},
the $k$-th column for $k \ge 3$  has a rational generating function which can be written
with denominator $(1-x^{k-2})(1-x^{k-1})(1-x^k)$.
\end{conj}

%                12      10        8        7      6       5           4      3         2
% t2 := - (4 x   - 4 x   + 8 x  + 2 x  - 9 x  - 11 x  - 13 x  - 2 x  + 5 x
%
 %                       /                  2                     3
 %   + 15 x + 9)  /  ((x + 1) (x  + x + 1) (x - 1) )
 %                    /

For example, column $3$ of Table~\ref{TabWangt}, the sequence $\{t_{n,3}\}$, appears
to have generating function
\beql{BCtriangles3}
x^3 ~ \frac{9+15 x+5 x^2 -2 x^3 -13 x^4 -11 x^5 -9 x^6 + 2 x^7 +8 x^8 -4 x^{10}   + 4 x^{12}}{(1-x)(1-x^2)(1-x^3)}\,.
 \eeq
%The conjecture would imply that the columns of all three tables 
%are holonomic sequences (cf. \cite[\S7.2]{KaPa11}) of degree $1$ and order $3k-6$ (if $k$ is even)
%or $3k-5$ if $k$ is odd.
It would be nice to know more about these quantities.

%%%%%%%%%%%%%%%%%%%%%%%%%%%%%%%%%%%%%%%%%%%
%%%   Section 5   THE NODES IN BC(1,n)
%%%%%%%%%%%%%%%%%%%%%%%%%%%%%%%%%%%%%%%%%%%

\section{The nodes in \texorpdfstring{$BC(1,n)$}{BC(1,n)}}\label{SecBCnodes}
Besides looking at the cells of $BC(1,n)$, it is also interesting to study the nodes.
For $n \ge 2$, $BC(1,n)$ has four boundary nodes of degree $n+1$ and
$2n-2$ boundary nodes of degree $n+2$.
An  interior node formed
   when $c$ chords (say) cross has degree $2c$.
   Let $v_{n,c}$ denote the number of interior nodes of degree $2c$, for $2 \le c \le n+1$.
   % n $BC(1,n)$ where $c$ chords cross.
   Table~\ref{TabWangi} shows the values of $v_{n,c}$ for $n \le 10$.
   A more extensive table, for $n \le 100$, is given in \seqnum{A333275}.
   
   \begin{table}[!ht] % vx degs A333275
\caption{Number $v_{n,c}$ of interior nodes  in $BC(1,n)$ where $c$ chords cross (\seqnum{A333275}). }\label{TabWangi}
$$
   \begin{array}{c|rrrrrrrrrr}
n \backslash c  & 2 & 3 & 4 & 5 & 6 & 7 & 8 & 9 & 10  & 11  \\
\hline
1 & 1 &&&&&&&&& \\
2 & 6 &  1 &&&&&&&& \\
3 & 24 &  2 &  1 &&&&&&&\\
4 & 54 &  8 &  2 &  1 &&&&&& \\
5 & 124 &  18 &  2 &  2 &  1 &&&&&\\
6 & 214 &  32 &  10 &  2 &  2 &  1 &&&& \\
7 & 382 &  50 &  22 &  2 &  2 &  2 &  1 &&& \\
8 & 598 &  102 &  18 &  12 &  2 &  2 &  2 &  1 && \\
9 & 950 &  126 &  32 &  26 &  2 &  2 &  2 &  2 &  1  & \\
10 & 1334 &  198 &  62 &  20 &  14 &  2 &  2 &  2 &  2 &  1 
   \end{array}
 $$
 \end{table}
\begin{thm}\label{Thvxdeg}
For $n \ge 2$, the numbers $v_{n,c}$ satisfy:
\begin{align}%\label{Eq3vxdeg1}
   \sum_{c=2}^{n+1} v_{n,c} ~+~ 2n+2  & ~=~ \sN(1,n)\,,  \label{Eqvxdeg1} \\
   \sum_{c=2}^{n+1} c \, v_{n,c} ~+~ n^2+4n+1  & ~=~ \sE(1,n)\,,   \label{Eqvxdeg2}\\
   \sum_{c=2}^{n+1}  \binom{c}{2}  \, v_{n,c}   & ~=~\binom{n+1}{2}^2\,.   \label{Eqvxdeg3}   
\end{align}
\end{thm}
\begin{proof}
The first equation simply gives the total number of nodes in $BC(1,n)$. For \eqn{Eqvxdeg2} we count
pairs $(\alpha, \beta)$,  where $\alpha$ is a cell and $\beta$ is a node, in two ways, obtaining
$$
3T(n)+4Q(n) ~=~ 4(n+1) + (2n-2)(n+2) + \sum_{c} 2c\, v_{n,c}\,,
$$
and use \eqn{EqTQ2}. To establish \eqn{Eqvxdeg3}, we start with the observation that if all the 
$2n+2$ boundary points of
$BC(1,n)$ are perturbed by small random amounts, there will be no triple or higher-order intersection points,
all the internal nodes will be simple, and there will be
$\binom{n+1}{2}^2$ of them (since any pair of nodes on the upper side of the rectangle and any pair of nodes on the lower side will determine a unique intersection point).
As the boundary points are returned to their true positions, the internal nodes coalesce.
If there is an interior point where $c$ chords intersect, the $\binom{c}{2}$ interior nodes there
coalesce into one, and we lose $\binom{c}{2}-1$ intersections.  We are left with the $\sN(1,n)-(2n+2)$
interior intersection points. Thus
$$
   \sum_{c=2}^{n+1} \left(\binom{c}{2}-1\right)v_{n,c} ~+~ \sN(1,c) -  (2n+2)   ~=~ \binom{n+1}{2}^2\,,   
$$
which simplifies to give \eqn{Eqvxdeg3}.
   \end{proof}
                           
However, we do not even have a formula for the number of simple
interior intersection points in $BC(1,n)$ (the first column of
Table~\ref{TabWangi}, the sequence 
$\{v_{n,2}\}$, \seqnum{A334701}), although we have computed $500$ terms. 
The first $100$ terms are shown in Table~\ref{TabA334701}.
%  It seems likely that a formula should exist.
We feel that a formula should exist!

\begin{table}[!ht] 
\caption{The first $100$ terms of the number of simple interior intersection points in $BC(1,n)$.}\label{TabA334701}
$$
   \begin{array}{rrrr}
\mbox{Terms~} $1--25$  & $26--50$ &  $51--75$ & $76--100$  \\
\hline
1 &   49246 &   679040 &   3264422  \\
6 &   57006 &   732266 &   3438642  \\
24 &   65334 &   790360 &   3616430  \\
54 &   75098 &   849998 &   3805016  \\
124 &   85414 &   914084 &   3998394  \\
214 &   97384 &   980498 &   4202540  \\
382 &   110138 &   1052426 &   4408406  \\
598 &   124726 &   1125218 &   4626162  \\
950 &   139642 &   1203980 &   4850198  \\
1334 &   156286 &   1285902 &   5085098  \\
1912 &   174018 &   1374300 &   5321854  \\
2622 &   194106 &   1463714 &   5571470  \\
3624 &   214570 &   1559064 &   5826806  \\
4690 &   237534 &   1657422 &   6095870  \\
6096 &   261666 &   1762004 &   6369534  \\
7686 &   288686 &   1869106 &   6655902  \\
9764 &   316770 &   1983922 &   6948566  \\
12010 &   348048 &   2102162 &   7256076  \\
14866 &   380798 &   2228512 &   7565826  \\
18026 &   416524 &   2356822 &   7889032  \\
21904 &   452794 &   2493834 &   8220566  \\
25918 &   492830 &   2635310 &   8568428  \\
30818 &   534962 &   2786090 &   8919298  \\
36246 &   580964 &   2938326 &   9285288  \\
42654 &   627822 &   3099230 &   9658638  
   \end{array}
 $$
 \end{table}

\begin{problem}\label{OPSIP}
Find a formula for the number of simple interior intersection points in $BC(1,n)$
(see Table~\ref{TabA334701} for $100$ terms, or \seqnum{A334701} for $500$ terms).
\end{problem}

%%%%%%%%%%%%%%%%%%%%%%%%%%%%%%%%%%%%%%%%%%%%%
%  FIG 13
%%%%%%%%%%%%%%%%%%%%%%%%%%%%%%%%%%%%%%%%%%%%%
%\begin{figure}[!ht]
%   \centerline{\includegraphics[angle=0, width=2.0in]{Fig32CfnGimp3.png}}
%   \caption{Comparison of the graphs $BC(3,2)$ (black lines),
%   $AC(3,2)$ (add the red lines), and $LC(3,2)$ (also add the blue lines).}
%   \label{Fig23Cfn}
%   \end{figure}

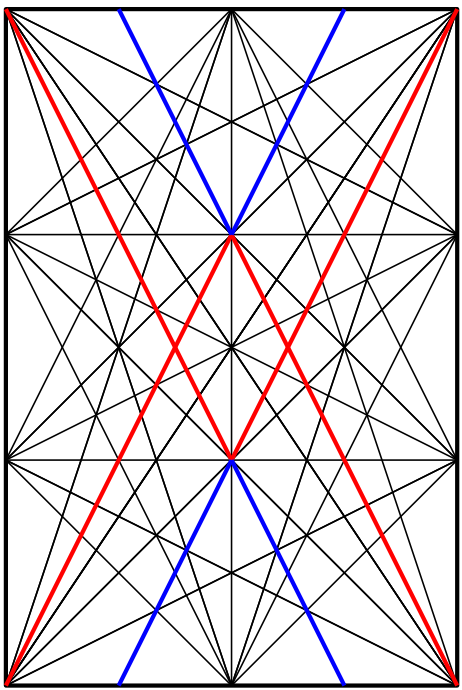
\begin{figure}[!ht]

\begin{center}
\begin{tikzpicture}[scale=3,black,semithick,line join=round]

\foreach \i in {0,...,3} {\foreach \j in {0,...,3} \draw (0,\i) -- (2,\j); }
\foreach \i in {0,...,3} {\foreach \j in {1,...,2} { \draw (0,\i) -- (\j, 0); \draw (0,\i) -- (\j, 3);} }
\foreach \i in {0,...,2} {\foreach \j in {0,...,2} \draw (\i,0) -- (\j,3); }
\foreach \i in {0,...,1} {\foreach \j in {1,...,3} { \draw (\i,0) -- (2,\j); \draw (\i,3) -- (2,\j);} }

\draw [ultra thick] (0,0) -- (2,0) -- (2,3) -- (0,3) -- (0,3) -- cycle;

\draw [red,ultra thick] (0,0) -- (1,2);
\draw [red,ultra thick] (2,0) -- (1,2);
\draw [red,ultra thick] (0,3) -- (1,1);
\draw [red,ultra thick] (2,3) -- (1,1);

\draw [blue,ultra thick] (0.5,0) -- (1,1);
\draw [blue,ultra thick] (1.5,0) -- (1,1);
\draw [blue,ultra thick] (0.5,3) -- (1,2);
\draw [blue,ultra thick] (1.5,3) -- (1,2);

\end{tikzpicture}
\end{center}

        \caption{Comparison of the graphs $BC(3,2)$ (black lines),
                $AC(3,2)$ (add the red lines), and $LC(3,2)$ (also add the blue lines).}
        \label{Fig23Cfn}
\end{figure}

%%%%%%%%%%%%%%%%%%%%%%%%%%%%%%%%%%%%%%
% Section: BC(m,n) - used to be called SC(m,n)  -- Rectangles
%%%%%%%%%%%%%%%%%%%%%%%%%%%%%%%%%%%%%%

\section{\texorpdfstring{$BC(m,n)$: $m \times n$}{BC(m,n): m X n} rectangular windows}\label{SecBCmn}
The graph $BC(1,n)$ ($n \ge 1$) is based on a $1 \times n$ rectangle.
In this section we consider what happens if we start more generally 
from an $(m,n)$-reticulated rectangle
(where $m \ge 1$, $n \ge 1$):
this is a rectangle of size $m \times n$ in which both vertical edges are divided into $m$ equal parts, and both horizontal edges into $n$ equal parts.
There are $m-1$ nodes on each vertical edge and
$n-1$ nodes on each horizontal edge, for a total of $4+2(m-1)+2(n-1) = 2(m+n)$ boundary nodes.

We will discuss three families of graphs based on these rectangles, denoted by $BC(m,n)$, $AC(m,n)$,
and $LC(m,n)$. The graph $BC(m,n)$ is formed by joining every pair of boundary nodes
by a line segment and placing a node at each point where two or more line segments intersect.
%Fig. \ref{FigBC11bw} is a drawing of $BC(1,1)$ and 
Figs. \ref{FigBC22bw} and \ref{FigBC22} show $BC(2,2)$, and Fig.~\ref{FigBC33} shows $BC(3,3)$.
(``$BC$'' stands for ``boundary chords''.)

Alternatively, we  could have constructed $BC(m,n)$ by starting with an $m \times n$ grid of 
equal squares, and then  joining each pair of boundary  grid points by a line segment.
However, if we include the interior grid points, there are there are $(m+1)(n+1)$ grid points in all,
and if we join \emph{each} pair of  grid points by a line segment, we obtain the
graph $AC(m,n)$.  
(``$AC$'' stands for ``all chords''.) 
These graphs are discussed by Huntington T. Hall \cite{Hall04},
Marc~E.~Pfetsch and G\"{u}nter~M.~Ziegler \cite{PfZi04}, and Hugo Pfoertner (entry \seqnum{A288187} in \cite{OEIS}). 
We shall say more about $AC(m,n)$  in \textsection\ref{SecAC}.

A third family of graphs, $LC(m,n)$, arises if we extend each line segment in $AC(m,n)$ 
until it reaches the boundary of the grid. 
(``$LC$'' stands for ``long chords''.)
These graphs are discussed by Seppo~Mustonen \cite{Mus08, Mus09, Mus10}.
We say more about $LC(m,n)$ in \textsection\ref{SecLC}.

Figure~\ref{Fig23Cfn} shows the differences between the three definitions in the case of a $(3,2)$ reticulated rectangle, the first time the definitions differ.
The black lines form the graph $BC(3,2)$. 
The four red lines are the additional line segments that appear when we construct $AC(3,2)$. They
start at an interior grid point and so are not present in $BC(3,2)$.
The four blue lines extend the red chords until they reach the boundary
of the grid, and form $AC(3,2)$.

%%%%%%%%%%%%%%%%%%%%%%%%%%%%%%%%%%%%%%%%%%%
%%%  FIG 14  BC33
%%%%%%%%%%%%%%%%%%%%%%%%%%%%%%%%%%%%%%%%%%%
%
%  The source for this is Scott's figure in A331452,
%H Scott R. Shannon, <a href="/A331452/a331452_1.png">Colored illustration for T(3,3)</a>
%
%%%  \begin{figure}
% width=4.8 produces a fig of actual screen size 5X5
%%%   \centerline{\includegraphics[angle=0, width=4.8in]{rectangle_3_by_3_7_small.png}}
% Aug 07 from Scott converted from gif to png     rectangle_3_by_3_7_small.png
%\centerline{\includegraphics[angle=0, width=7in]{rectangle_3_by_3_7.png}}
%\centerline{\includegraphics[angle=0, width=6in]{SC33c.png}}
%\centerline{\includegraphics[angle=0, width=1.2in]{SC.4.1s2Gimp2.png}}
%\centerline{\includegraphics[angle=90, width=6.0in]{SC.4.1s2Gimp2.png}}
 % \centerline{\includegraphics[angle=90, width=6.0in]{tmp1.gif}}
  %%%   \caption{The graph $BC(3,3)$. There are $257$ nodes and $340$ cells. The cells are  colored according to the ... (fill in later).}
   %%%  \label{FigBC33}
   %%%  \end{figure}
   
%%   Lars's new BC33 figure follows
   
 \begin{figure}[!ht]
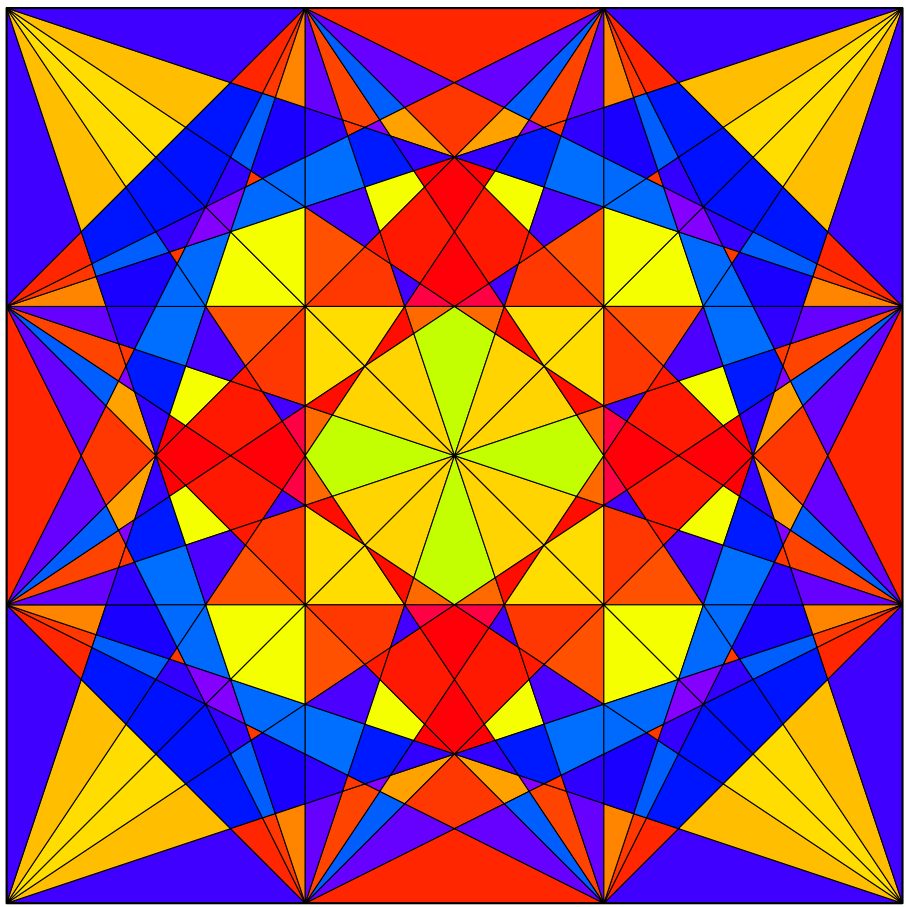

	\begin{center}

	\end{center}
	  \caption{The graph $BC(3,3)$. There are $257$ nodes and $340$ cells.}
   \label{FigBC33}
%\caption{ Random seed=3006 | H=5 Dev=20 Mult=1.2 | H=35 Dev=20 Mult=1.2 | H=50 Dev=5 Mult=0.5 | H=120 Dev=5 Mult=0.2 | H=240 Dev=15 Mult=2}
\end{figure}
%  end of Lars's new BC33 figure

The numbers of nodes  $\sN(m,n)$ 
 and cells $\sC(m,n)$ in $BC(m,n)$  are shown for $m, n \le 37$ in 
\seqnum{A331453} and \seqnum{A331452}, respectively,
and the initial terms are shown in Table~\ref{TabBC1}.

%%%%%%%%%%%%%%%%%%%%%%%%%%%%%%%%%
%%% TABLE 7
%%%%%%%%%%%%%%%%%%%%%%%%%%%%%%%%%%%%%
\begin{table}[!htb]
\caption{Numbers of nodes $\sN(m,n)$ and cells $\sC(m,n)$ in $BC(m,n)$ for $1 \le m, n \le 7$.}\label{TabBC1}
$$
\begin{array}{|c|c|c|c|c|c|c|c|}
\hline 
m \backslash \, n & 1 & 2 & 3 & 4 & 5 & 6 & 7  \\
\hline 
1 & 5, 4 & 13, 16 & 35, 46 & 75, 104 & 159, 214 & 275, 380 & 477, 648   \\
2 & 13, 16 & 37, 56 & 99, 142 & 213, 296 & 401, 544 & 657, 892 & 1085, 1436   \\ 
3 & 35, 46 & 99, 142 & 257, 340 & 421, 608 & 881, 1124 & 1305, 1714 & 2131, 2678   \\
4 & 75, 104 & 213, 296 & 421, 608 & 817, 1120 & 1489, 1916 & 2143, 2820 & 3431, 4304   \\
5 & 159, 214 & 401, 544 & 881, 1124 & 1489, 1916 & 2757, 3264 & 3555, 4510 & 5821, 6888   \\
6 & 275, 380 & 657, 892 & 1305, 1714 & 2143, 2820 & 3555, 4510 & 4825, 6264 & 7663, 9360   \\ 
7 & 477, 648 & 1085, 1436 & 2131, 2678 & 3431, 4304 & 5821, 6888 & 7663, 9360 & 12293, 13968 \\  
\hline
\end{array}
$$
\end{table}
% new table end

Regrettably, except when $m$ or $n$ is $1$, we have been unable to find formulas
for any of these quantities.
The diagonal case, when $m=n$, is the most interesting (because
the most symmetrical), but is also probably the hardest to solve.
In accordance with our philosophy of ``if you can't solve it, make art'',
Fig.~\ref{FigBC33}  shows our stained glass window $BC(3,3)$, and 
entry \seqnum{A331452} has a large number of larger and even more striking  examples which space
restrictions do not permit us to show here.

%%%%%%%%%%%%%%%%%%%%%%%%%%%%%%%%%%%%%%%%%%%%%%%
 %%%%   FIG 15  BC42
 %%%%%%%%%%%%%%%%%%%%%%%%%%%%%%%%%%%%%%%%%%%%%%
%\begin{figure}[!ht]
%\centerline{\includegraphics[angle=0, width=2.4in]{BC42Gimp.png}}
%\caption{$BC(4,2)$ with cells color-coded to distinguish triangles (red), quadrilaterals (yellow), and pentagons (blue).}
%\label{FigBC42}
%\end{figure}

\input{InputBC42}

Out of all these unsolved problems, the case $m=2$ (or $n=2$) would seem
to be the most amenable to analysis, 
perhaps by extending the work of Legendre  \cite{Leg09} and  Griffiths \cite{Gri10}.
For instance, what are the conditions
for three chords in $BC(2,n)$ to intersect at a common point?
We emphasize this by stating:

\begin{problem}\label{OPm=2}
Find formulas for the numbers of nodes $(\sN(2,n)$, \seqnum{A331763}$)$ and 
cells $(\sC(2,n)$, \seqnum{A331766}$)$ in $BC(2,n)$.
\end{problem}

The first $10$ terms are given in Table~\ref{TabBC2}, and $100$ terms are given in the 
entries for these two sequences  in~\cite{OEIS}.

%%%%%%%%%%%%%%%%%%%%%%%%%%%%%%%%%%%%%%%%%%%%%%%
%%% TABLE 8   NODES AND CELLS IN BC(2,n)
%%%%%%%%%%%%%%%%%%%%%%%%%%%%%%%%%%%%%%%%%%%%%%
\begin{table}[!ht]
\caption{Numbers of nodes and cells in $BC(2,n)$. }\label{TabBC2}
$$
   \begin{array}{crrrrrrrrrrrc}
   n: & 1 & 2 & 3 & 4 & 5 & 6 & 7 & 8 & 9 & 10 &  \cdots &  \cite{OEIS} \\
    \sN(2,n): & 13 &  37 &  99 &  213 &  401 &  657 &  1085 &  1619 &  2327 &  3257 &   \cdots &  \seqnum{A331763}  \\
    \sC(2,n): & 16 &  56 &  142 &  296 &  544 &  892 &  1436 &  2136 &  3066 &  4272  &\cdots & \seqnum{A331766}
   \end{array}
 $$
 \end{table}

%%%%%%%%%%%%%%%%%%%%%%%%%%%%%%%%%
%%%%  New subsection about count of cells in BC(2,n) by number of sides
%%%%%%%%%%%%%%%%%%%%%%%%%%%%%%%%%

In $BC(1,n)$ the cells are always triangles or quadrilaterals (Theorem~\ref{ThABZ}).
It appears that a similar phenomenon holds for $BC(2,n)$. 
The data strongly suggests the following conjecture.

\begin{conj}\label{ConjBC2cell}
The cells in $BC(2,n)$ have at most eight sides, and for $n \ge 19$, at most six sides.
\end{conj}

We have verified the conjecture for $n \le 106$.
Row $n$ of Table~\ref{TabBC2cells} gives the number of cells in $BC(2,n)$ with $k$ sides,
for $k \ge 3$ and $n \le 20$.  For rows $n=1, 2,$ and $3$ of this table  see Figs.~\ref{FigBC12},
\ref{FigBC11bw}, and \ref{Fig23Cfn} (black lines only). 
For row $4$ see Figure~\ref{FigBC42}, where one can see  that  $BC(4,2)$ 
has $192$ triangular cells (red), $92$ quadrilaterals (yellow),
and $12$ pentagons (blue). 
Entry  \seqnum{A335701} gives the first $106$ rows of this table, and
has many further illustrations.
The row sums  in Table~\ref{TabBC2cells} are the numbers $\sC(2,n)$ given
 in column $2$ of Table~\ref{TabBC1} and \seqnum{A331766}.

%%%%%%%%%%%%%%%%%%%%%%%%%%%%%%%
%  TABLE 9:   The cells counts for BC(n,2) New, added August 15 2020
%%%%%%%%%%%%%%%%%%%%%%%%%%%%%%%
\begin{table}[!htb]
\caption{Row $n$ gives the number of cells in $BC(2,n)$ with $k$ sides,
for $k \ge 3$. It appears that for $n \ge 19$, no cell has more than six sides (see \seqnum{A335701}). }\label{TabBC2cells}
$$
\begin{array}{c|cccccc}
%\hline
n \backslash \, k & 3 & 4 & 5 & 6 & 7 & 8 \\
\hline
1 &  14 &  2 \\
2 &  48 &  8 \\
3 &  102 &  36 &  4 \\
4 &  192 &  92 &  12 \\
5 &  326 &  194 &  24 \\
6 &  524 &  336 &  28 &  4 \\
7 &  802 &  554 &  80 \\
8 &  1192 &  812 &  128 &  4 \\
9 &  1634 &  1314 &  112 &  0 &  4 &  2 \\
10 &  2296 &  1756 &  200 &  20 \\
11 &  3074 &  2508 &  236 &  22 \\
12 &  4052 &  3252 &  356 &  28 \\
13 &  5246 &  4348 &  472 &  28 \\
14 &  6740 &  5464 &  652 &  28 \\
15 &  8398 &  7054 &  656 &  74 \\
16 &  10440 &  8760 &  940 &  52 \\
17 &  12770 &  11050 &  1040 &  58 \\
18 &  15512 &  13324 &  1300 &  60 &  4 \\
19 &  18782 &  16162 &  1600 &  70 \\
20 &  22384 &  19256 &  1948 &  104 \\
\hline
\end{array}
$$
\end{table}

More generally we may ask:
For $BC(m,n)$, $m$ fixed, is there an upper bound on the number of sides of a cell as $n$ varies?

%%%%%%%%%%%%%%%%%%%%%%%%%%%%%%%%%%%%
%%% FIG 16  DISSECTION OF CORNER
%%%%%%%%%%%%%%%%%%%%%%%%%%%%%%%%%%%%%%

We are at least able to analyze the corner squares of $BC(2,n)$.

\begin{thm}\label{ThBCcnr2}
For $n=2$ the four corner squares of $BC(2,n)$ (and $BC(n,2)$)  each contain $12$ triangles 
and $4$ quadrilaterals, while for $n=3$ 
they contain $15$ triangles, $6$ quadrilaterals, and (exceptionally) one pentagon.
For $n \ge 4$, the  corner squares
each contain $7n+1$ cells,
consisting of $2n+9$ triangles and $5n-8$ quadrilaterals.
\end{thm}

\begin{proof}
We consider the top left corner square of $BC(n,2)$, assuming $n \ge 4$.
The key to the proof is to dissect this square into regions, in each of which the cell
structure is apparent, and such that the boundaries of the regions do not cross any cell boundaries.
This is done as indicated in Fig.~\ref{FigBCcnr}. There are six regions, labeled $a$ through $f$.

We assume the coordinates are chosen so that nodes $A, B, C, D$ have coordinates
$(0,0)$, $(1,0)$, $(1,1)$, and $(0,1)$, respectively. The four vertices
of the rectangle defining $BC(n,2)$ have coordinates $(0,0)$, $(2,0)$, $(2,n)$, and $(0,n)$.

The chord from $A$ to the grid point $(1,n)$ cuts the line $DC$ midway between $D$ and $F$, 
and the  $n-1$ chords from $A$ to grid points $(2,n)$, $(2,n-1), \ldots$, $(2,2)$ cut $DC$ between $F$ and $C$.
The final chord from $A$ to $(2,1)$ cuts $BC$ at $E$.
 The top left triangular region $f$ is therefore divided into $n+2$ triangular cells.
 
 There is a chord from $B$ to $D$, a chord from $B$ to the grid point $(0,2)$ which cuts $DC$  at $F$, 
 and  $n-2$  further
 chords from $B$ to the grid points $(0,3), \ldots, (0,n)$, which
 cut $DC$ to the right of $F$.
 
 There is one further chord that cuts this corner square, the chord from $D$ to $E$
 
 The reader will now have no difficulty in verifying that the cells in regions $a,b,c,d,e,f$ are 
 as shown in Table~\ref{Tababcdef}.
 
 \begin{figure}[!ht] 
\begin{center}
\begin{tikzpicture}[scale=1]
\draw (0,0) -- (0,6);
\draw (0,0) -- (6,0);
\draw (0,0) -- (6,6);
\draw (0,6) -- (6,6);
\draw (6,0) -- (6,6);
\draw (3,0) -- (2,2);
\draw (3,0) -- (4,2);
\draw (3,3) -- (6,0);
\draw ((4,4) -- (6,3);

\node [above] at (0,6) {A};
\node [above] at (6,6) {B};
\node [below] at (6,0) {C};
\node [below] at (0,0) {D};
\node [right] at (6,3) {E};
\node [below] at (3,0) {F};

\node at (1.5,0.65) {e};
\node at (3,1.8) {d};
\node at (4.5,0.65) {c};
\node at (4.8,2.6) {b};
\node at (5.4,4.3) {a};
\node at (2,4) {f};
\end{tikzpicture}
\caption{Dissection of corner square of $BC(n,2)$, $n \ge 4$, used in proof of Theorem~\ref{ThBCcnr2}.}
\label{FigBCcnr}
\end{center}
\end{figure}
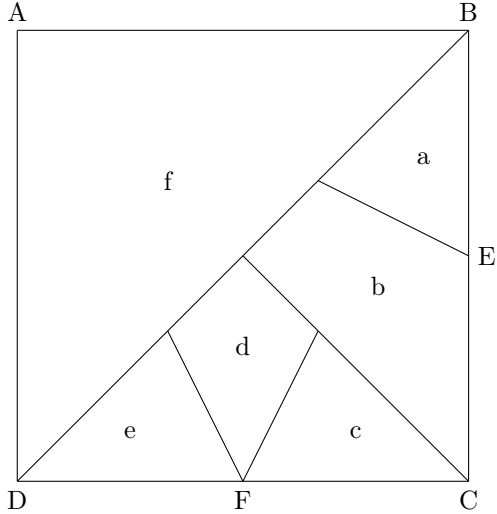

 \begin{table}[!htb]
  \caption{Numbers of triangles and quadrilaterals in the regions shown in Fig.~\ref{FigBCcnr}. }\label{Tababcdef}
 $$
 \begin{array}{ccc}
 \mbox{Region }& \mbox{triangles} & \mbox{quadrilaterals} \\
 \hline 
 a & n & 0  \\
 b & 2 & 2n-3 \\
 c & 2 & n-2 \\
 d & 1 & 3 \\
 e & 2 & 2n-6 \\
 f & n+2 & 0  \\
 \hline\mbox{Total} & 2n+9 & 5n-8 \\
\end{array}
$$
\end{table}
\end{proof}

\section{\texorpdfstring{$BC(m,n)$}{BC(m,n)} in general position.}\label{SecBCGP}

We can  obtain reasonably good upper bounds on
$\sN(m,n)$ and $\sC(m,n)$  by analyzing what would happen if all the intersection 
points in $BC(m,n)$ were simple intersections---that is, if there was no interior 
point where three or more chords met.  

We use $BC_{GP}(m,n)$ to denote a graph obtained by perturbing the boundary nodes 
of $BC(m,n)$ (excluding the four vertices)  by small random sideways displacements along the boundaries.
That is, if a boundary node was a fraction $\frac{i}{j}$ say of the way along an edge, we move it to
a point $\frac{i}{j} + \epsilon$ of the way along the edge, where $\epsilon$ is a small random real number.
If the $\epsilon$'s are chosen independently, the new graph will be in ``general position", 
and there will be no multiple intersection points in the interior.

To illustrate the perturbing process, in Fig.~\ref{FigBCGP12} below one can see
 (ignoring for now the supporting strut on the left) a perturbed version of  $BC(1,2)$ obtained
 %an example of a $BC_{GP}(1,2)$ obtained
 by slightly displacing just one node (labeled $4$)  so as to avoid the triple
 intersection point at the center (see Fig.~\ref{FigBC12}).

Let $\sN_{GP}(m,n)$ and $\sC_{GP}(m,n)$ denote
the numbers of nodes and cells, respectively, 
in the perturbed graph.  The perturbations increase the numbers of nodes and cells, 
so $\sN_{GP}(m,n) \ge \sN(m,n)$
and $\sC_{GP}(m,n) \ge \sC(m,n)$, and
$\sN_{GP}(m,n) \to \sN(m,n)$
and $\sC_{GP}(m,n) \to \sC(m,n)$ as the displacements are reduced to zero.

\begin{thm}\label{Th129.1}
For $m,n \ge 1$, the number of interior nodes in $BC_{GP}(m,n)$ is
\beql{EqIGPmn}
 \frac{1}{4}\, \lbrace (m+n)(m+n-1)^2(m+n-4) ~+~ 2mn(2m+n-1)(m+2n-1)\rbrace\,. 
\eeq
\end{thm}
\begin{proof}We start with the observation that any four boundary nodes of the rectangle, no three of which are on an edge, determine a unique intersection point in the interior of the rectangle.  There are several ways to choose these 
four points.  They might be the four vertices of the rectangle, which can be done in just one way.
They might consist of three vertices and a single node on one of the other two sides,
which can be done in $4(m_1+n_1)$ ways, where $m_1 = m-1$ and $n_1 = n-1$ are the numbers of ways of choosing a single non-vertex point on a side. A more typical example 
consists of one vertex, and one, resp. two, points on the two opposite sides, as shown in the following drawing.  This can be done in $4(m_1 n_2 + m_2 n_1)$ ways, where 
$m_2 = (m-1)(m-2)/2$, $n_2 = (n-1)(n-2)/2$ are the numbers of ways of choosing two non-vertex nodes from the sides.
% %%% little unnumbered fig in text FIG 15bis
\begin{center}
\begin{tikzpicture}[scale=2]
\coordinate (P0) at (0,0);
\draw (0,0) -- (0,1);
\draw (0,0) -- (1.5,0);
\draw (0,1) -- (1.5,1);
\draw (1.5,0) -- (1.5,1);
\draw[fill] (0,1) circle[radius = 0.04];
\draw[fill] (1,0) circle[radius = 0.04];
\draw[fill] (1.5,0.5) circle[radius = 0.04];
\draw[fill] (1.5,0.7) circle[radius = 0.04];
\draw [thick] (1.0,0) -- (1.5,0.7);
\draw [thick] (0,1.0) -- (1.5,0.5);
%\node at (-.3,.5) {$m-1$};
\node at (1.0,-.2) {$n-1$};
%\node at (1.8,0.65) {$(m-1)(m-1)/2$};
\node at (1.9,0.65) {$\binom{m-1}{2}$};
\end{tikzpicture}
\end{center}

There are in all seventeen different configurations for choosing four points, and 
when the seventeen counts are added up the result is the expression given in \eqn{EqIGPmn}. We omit the details.
\end{proof}

\vspace*{+.2in}
   \noindent{\bf Remarks.}
(i) Since there are $2(m+n)$ boundary nodes, the total number of
nodes in $BC_{GP}(m,n)$ is
\beql{EqNGPmn}
\sN_{GP}(m,n) ~=~ \frac{1}{4}\, \lbrace (m+n)(m+n-1)^2(m+n-4) ~+~ 2mn(2m+n-1)(m+2n-1)\rbrace +2(m+n)\,. 
\eeq
This is our upper bound for $\sN(m,n)$.

(ii) Another way to interpret $\sN_{GP}(m,n)$ is that this is the number of nodes in $BC(m,n)$ 
counted with multiplicity (meaning that if there is an interior
node where $c$ chords meet, it  contributes $c-1$ to the total).

(iii) When $m=n$, \eqn{EqNGPmn} simplifies to
\beql{EqNGPnn}
\frac{n}{2}\,(17 n^3-30 n^2+19 n+4)\,,
\eeq
%and so we have
%\beql{EqGPnn2}
which is our upper bound for  $\sN(n,n)$. % \le \frac{n}{2} (17 n^3-30 n^2+19 n+4)\,.
For $n=52$, $\sN(n,n) = 52484633$ (from \seqnum{A331449}), while \eqn{EqNGPnn} 
gives $60065408$,  too large by a factor of $1.14$, which is not too bad.
The moral seems to be that most internal nodes are simple.

(iv) When $m=1$, \eqn{EqNGPmn}  becomes $n^2(n+1)^2/4$, which agrees with the number mentioned in the proof of Theorem~\ref{Thvxdeg}.

(v) For large $m$ and $n$, the expression \eqn{EqNGPmn} is dominated by the degree $4$ terms, which are
\beql{EqDeg4}
\frac{1}{4} \left( m^4 + n^4 + 8mn(m^2+n^2) + 16 m^2 n^2\right)\,.
\eeq
Setting $m=n$, we get  $\sN_{GP}(n,n)  \sim 17n^4/2$ as $n \to \infty$.  We can confirm this by looking
at the number of ways to choose four nodes out
of the $4n$ boundary nodes so that no three are on a side. This is (essentially)
$$
\binom{4n}{n} ~-~ 4\,\binom{n}{4} ~-~12\, n\, \binom{n}{3} ~\sim~ \frac{17}{2}n^4\,.
$$

(vi) From (v), we have $\sN(n,n) =  O(n^4)$. In fact, we conjecture that $\sN(n,n) \sim \sN_{BC}(n,n) \sim 17\,n^4/2$.
But to establish this we would need better information about the number of
interior nodes in $BC(n,n)$ with a given multiplicity.

%%%%%%%%%%%%%%%%%%%%%%%%%%%%%%%%%%%%%%%%%
%%   Freeman  FIG 17
%%%%%%%%%%%%%%%%%%%%%%%%%%%%%%%%%%%%%%%%%

 \begin{figure}[!ht] 
\begin{center}
\begin{tikzpicture}[scale=3]
\coordinate (P0) at (-1.049770328, 1.974330837);
\coordinate (P1) at (-0.2380969460, 1.396741441);
\coordinate (P2) at (-1.638304089, 1.147152873);
\coordinate (P3) at (0.5735764361, 0.8191520445);
\coordinate (P4) at (-0.8191520445, 0.5735764365);
\coordinate (P4p) at (-0.9829824534, 0.6882917238);
\coordinate (P5) at (0,0);
\coordinate (P6) at (-1.638304089,0);

\draw (-2,0) -- (1,0);
\draw (P5) -- (P2);
\draw (P5) -- (P0);
\draw (P5) -- (P1);
\draw (P5) -- (P3);
\draw (P3) -- (P0);
\draw (P3) -- (P2);
%\draw (P3) -- (P4);
\draw (P3) -- (P4p);
\draw (P0) -- (P2);
\draw (P0) -- (P4p);
\draw (P1) -- (P2);
\draw (P1) -- (P4p);
\draw [line width=0.1cm]  (P6) -- (P2);

\node [above] at  (-1.049770328, 1.974330837)  {0};
\node [above] at  (P1)  {1};
\node [left] at  (P2)  {2};
\node [above] at  (P3)  {3};
\node [below] at  (P4p)  {4};
\node [below] at  (P5)  {5};

\draw [ultra thick, red] (-1.8, 2.2) -- (0.7, 2.2);
\draw [dashed, ultra thick, red] [->] (-1.8, 2.2) -- (-1.8, 2.0);
\draw [dashed, ultra thick, red] [->] (0.7, 2.2) -- (0.7, 2.0);

\end{tikzpicture}
\caption{$BC(1,2)$ in general position: node $4$ has been displaced slightly so as to avoid the triple intersection point at the center. The
strut on the left tilts the figure so that the ordinates of the boundary nodes 
are in the same order as the labels.  The red line is the ``counting line'', which descends across the picture in order to count the cells.}
\label{FigBCGP12}
\end{center}
\end{figure}
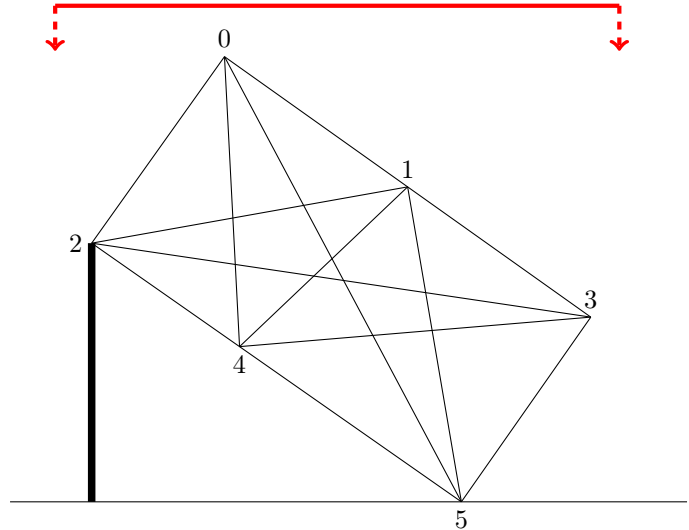

 \vspace{0.1in}
Now that we know the number of nodes, we can also find the number $\sC_{GP}(m,n)$  of cells in $BC_{GP}(m,n)$.
For this we use a method described by Freeman~\cite{Fre75}. The following is a slight modification of 
his procedure. 
%Here $G$  will denote the graph $BC{GP}(m,n)$. It has $2(m+n)$ boundary nodes.
%Without loss of generality we assume $m \le n$.
$BC_{GP}(m,n)$ has $2(m+n)$ boundary nodes.
We label the top left corner vertex $0$, and the bottom right corner vertex $2(m+n)-1$.
The nodes along the top edge we label $0$, $1$, $3$, $5 \ldots$,$2n-1$,  continuing along along the right-hand edge with $2n+1$, $2n+3 \ldots$, $2(m+n)-1$.
Along the left-hand edge we place the labels  $0$, $2$, $4, \ldots$, $2m$,
 continuing  along the bottom edge with $2m+2$, $2m+4, \ldots$, $2m+2n-2$, $2(m+n)-1$.

We now raise the bottom left corner of the rectangle until the boundary nodes are at 
different heights, and so that the order of the heights matches the order of the labels
(node $0$ becomes the highest point, followed by nodes $1$, $2, \ldots$ in order).
Fig.~\ref{FigBCGP12} illustrates the case $BC_{GP}(1,2)$. The black strut raises 
the bottom left corner so that the heights of the nodes are in the correct order.

We now take a horizontal line (Freeman calls it a ``counting line''),
and slide it downwards from the top of the figure to
the bottom, recording each time it cuts a new cell.
The counting line is shown in red in the figure.

When the counting line reaches a boundary node, with label $k$ (say), 
the count is increased by the number of cells originating at $k$ that have
not yet been counted. This number is equal to the number
of boundary nodes with label greater than $k$ which are not on the same
side as $k$.
  On the other hand, when the counting line reaches an interior node the count increases by  exactly $1$ (this
  is because  there is no point where three chords meet). So
  the contribution to the count from the interior nodes is simply the number
  of interior nodes, which is known from Theorem~\ref{Th129.1}.

In Fig.~\ref{FigBCGP12}, the count goes up by $3$ at node $0$, by $3$ at node $1$,
$1$ at node $2$, and $1$ at node $3$, for a subtotal of $8$.  There are $9$ interior nodes, 
so the total number of cells is $17$.  

From a careful study of a tilted version of general case $BC_{GP}(m,n)$,, 
combined with~\eqn{EqNGPmn}, we obtain:

\begin{thm}\label{ThCGPmn}
For $m,n \ge 1$,  the number of cells in $BC_{GP}(m,n)$ is
\beql{EqCGPmn}
\sC_{GP}(m,n) ~=~ \frac{1}{4}\, \lbrace (m-1)^2(m-2)^2 + (n-1)^2(n-2)^2\rbrace 
+ 2mn\left(m+n-\frac{3}{2}\right)^2 + \frac{9mn}{2} -1\,.
\eeq
\end{thm}

\vspace*{+.2in}
   \noindent{\bf Remark.}
Asymptotically, $\sC_{GP}(m,n)$ and $\sN_{GP}(m,n)$ behave 
in the same way.  In fact the difference $\sC_{GP}(m,n)~-~\sN_{GP}(m,n)$ is only
a quadratic function of $m$ and $n$,
${m}^{2}+4\,mn+{n}^{2}-4\,m-4\,n+1$.

%%%%%%%%%%%%%%%%%%%%%%%%%%%%%%%%%%%%%%%%%%%
%%%  FIG 18 TIKZ WITH NESTED LOOPS AC(3,3)
%%%%%%%%%%%%%%%%%%%%%%%%%%%%%%%%%%%%%%%%%%%
%
%\begin{figure}
%\centerline{\includegraphics[angle=0, width=2in]{AC33a.png}}
%\centerline{\includegraphics[angle=0, width=1.2in]{SC.4.1s2Gimp2.png}}
%\centerline{\includegraphics[angle=90, width=6.0in]{SC.4.1s2Gimp2.png}}
 % \centerline{\includegraphics[angle=90, width=6.0in]{tmp1.gif}}
  % \label{FigAC33}
   %\end{figure}

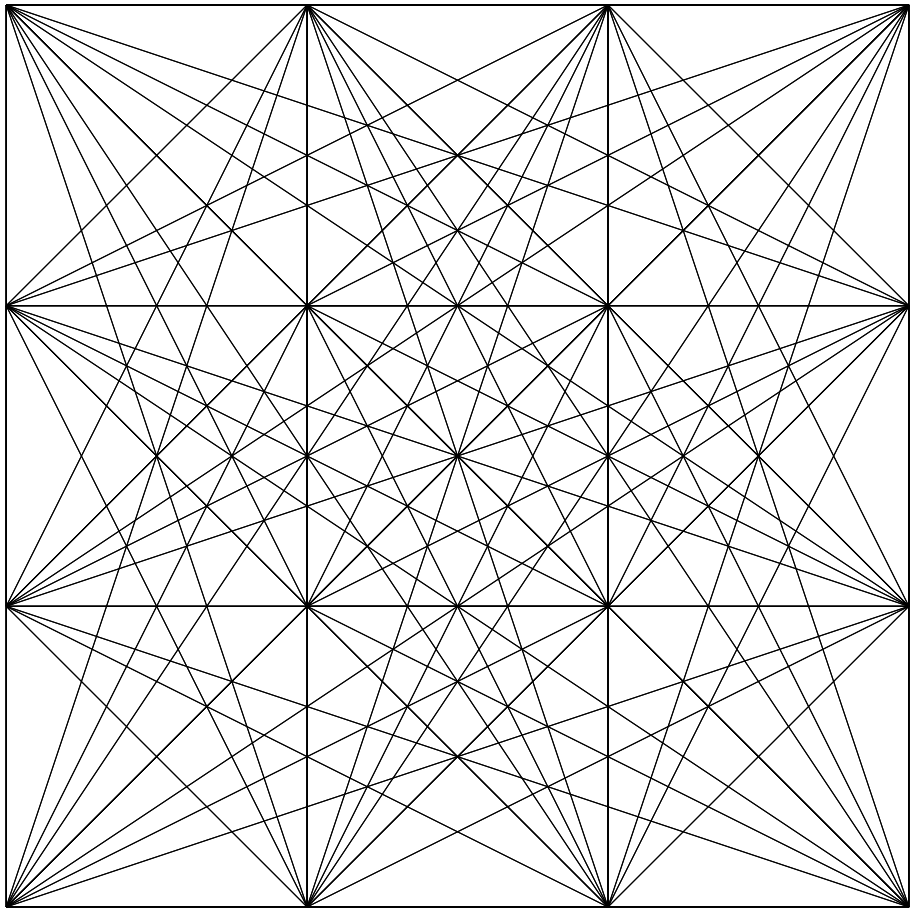
\begin{figure}[!ht]
\begin{center}
\begin{tikzpicture}[scale=4]
\foreach \i in {0,...,3}
{ \foreach \j in {0,...,3}
  { \foreach \k in {0,...,3} 
     { \foreach \m in {0,...,3}
  \draw (\i,\j) -- (\k,\m); }}}
\end{tikzpicture}
\caption{The graph $AC(3,3)$. There are $353$ nodes and $520$ cells.}
\label{FigAC33}
\end{center}
\end{figure}

%%%%%%%%%%%%%%%%%%%%%%%%%%%%%%%%%%%%%%
% Section AC
%%%%%%%%%%%%%%%%%%%%%%%%%%%%%%%%%%%%%%
\section{The graphs \texorpdfstring{$AC(m,n)$}{AC(m,n)}.}\label{SecAC}
The graph $AC(m,n)$ was defined in \textsection\ref{SecBCmn}.
We take an $(m+1) \times (n+1)$ square grid of nodes,
and draw a chord between \emph{every} pair of grid nodes. (If we only joined pairs
of boundary nodes we would get $BC(m,n)$.)

Figure~\ref{Fig23Cfn} shows $AC(3,2)$ (take the black and red lines only, not the blue lines).
Hugo Pfoertner has made black and white drawings of $AC(m,n)$
for $1 \le m, n \le 5$ in \seqnum{A288187}.
Figure~\ref{FigAC33} shows a black and white drawing of $AC(3,3)$ made using \emph{TikZ} 
\cite{Cre11, Tan10}.

The numbers of nodes $\sN_{AC}(m,n)$  and cells $\sC_{AC}(m,n)$ in $AC(m,n)$
are given for $m, n \le 9$ in \seqnum{A288180} and \seqnum{A288187}, respectively, and the initial terms are shown
in Table~\ref{TabAC1}.
The first row and column of Table~\ref{TabAC1} are the same as the first row and
 column of Table~\ref{TabBC1}
 %  (and of Table~\ref{TabLC1} below) 
 but are included for completeness.

% AC table start  A288180 and A288187,
\begin{table}[!htb]
\caption{Numbers of nodes $\sN_{AC}(m,n)$ and cells $\sC_{AC}(m,n)$ in $AC(m,n)$ for $1 \le m, n \le 7$.}\label{TabAC1}
$$
\begin{array}{|c|c|c|c|c|c|c|c|}
\hline 
m \backslash \, n & 1 & 2 & 3 & 4 & 5 & 6 & 7  \\
\hline 
1 & 5, 4& 13, 16& 35, 46& 75, 104& 159, 214& 275, 380& 477, 648 \\
2 & 13, 16& 37, 56& 121, 176& 265, 388& 587, 822& 1019, 1452& 1797, 2516 \\
3 & 35, 46& 121, 176& 353, 520& 771, 1152& 1755, 2502& 3075, 4392& 5469, 7644 \\
4 & 75, 104& 265, 388& 771, 1152& 1761, 2584& 4039, 5700& 7035, 9944& 12495, 17380 \\
5 & 159, 214& 587, 822& 1755, 2502& 4039, 5700& 8917, 12368& 15419, 21504& 27229, 37572 \\
6 & 275, 380& 1019, 1452& 3075, 4392& 7035, 9944& 15419, 21504& 26773, 37400& 47685, 65810 \\
6 & 477, 648& 1797, 2516& 5469, 7644& 12495, 17380& 27229, 37572& 47685, 65810& 84497, 115532 \\
\hline
\end{array}
$$
\end{table}
% AC table end

It is clear (compare Figs.~\ref{FigBC33} and \ref{FigAC33}) that $AC(m,n)$ contains far more nodes and 
cells than $BC(m,n)$. We may obtain an upper bound on $\sN_{AC}(n,n)$ as follows.
The graph $AC(n,n)$ has $(n+1)^2$ grid points. The number of ways
of choosing four grid points is $\binom{(n+1)^2}{4}$, and except for a vanishingly small fraction of cases,
no three points will be collinear. There are then two possibilities: the four points may form a convex 
quadrilateral, or a triangle with the fourth point in its interior.
In the first case the intersection of the two diagonals of the quadrilaterals is a 
node of $AC(m,n)$ (which may or may not be a new node), but in the second case no new node is formed.

If four points in the plane are chosen at random from a square, by what is known as ``Sylvester's Theorem'', 
the probability that they form a convex quadrilateral is $25/36$ and the probability that
they form a triangle with an interior point is $11/36$ (see \cite[Table~4]{Pfi89},  \cite[Table~3, p.~114]{Sol78}
for the complicated history of this result).
Assuming that Sylvester's theorem applies to our problem,
we can conclude that the number of nodes in $AC(m,n)$ counted
with multiplicity is asymptotically 
\beql{EqACN1}
\frac{25}{36} \binom{(n+1)^2}{4} ~\sim~ \frac{1}{35.56}\,n^8\,.
\eeq

Both Tom Duff (personal communication) and Keith F. Lynch 
(personal communication) have have carried out extensive experiments, studying what happens when 
four points are chosen from  an $m \times n$ grid,
and have found that there is excellent agreement with the predictions of Sylvester's Theorem.

In a remarkable calculation, Tom Duff enumerated and classified all sets of four points
chosen from an $m \times n$ grid for $m, n \le 349$.
In a $349 \times 349$  grid, there are 
$6366733094048270910$ strictly convex
quadrilaterals out of $9170030499095875150$ total.
The fraction is 0.6942979, just a little short of Sylvester's $25/36 = 0.694444\ldots$.
The deficit is explained by the not quite negligible counts of quadrilaterals
with at least three collinear points.
If those are included 
 with the strictly convex quadrilaterals, the ratio is
0.6945982, slightly more than 25/36.
This is convincing evidence that Sylvester's theorem does apply to our problem.

In any case, $\sN_{AC}(n,n) = O(n^8)$, compared with $\sN(n,n) = O(n^4)$ for $BC(n,n)$.

 %%%%%%%%%%%%%%%%%%%%%%%%%%%%%%%%%%%%%%%%%%%
%%%  FIG 19  LC(3,3) Originally Gimp of Scott's then replaced by Lars TIKZ the third of his fis
%%%%%%%%%%%%%%%%%%%%%%%%%%%%%%%%%%%%%%%%%%%
%
%  The source for this is Scott's figure in A333282,
%  https://oeis.org/A333282/a333282_2.png
%
%%%%   \begin{figure}
% width=4.8 produces a fig of actual screen size 5X5
%%%%   \centerline{\includegraphics[angle=0, width=4.8in]{a333282_2Gimp1.png}}
%%%    \caption{The graph $LC(3,3)$. There are $405$ nodes and $624$ cells. The cells are  colored according to the ... (fill in later).}\label{FigLC33}
%%%%   \end{figure}

% Lars's Fig 19
% This is the third of his figures from the file LarsAC33b.tex,  consisting of the
% commands on lines 1407-2080 (675 lines)

% Standard width for figs of the BC(3,3) family, when viewed in the Tex Shop Pdf built-in viewer
% at scale = 161 is 5"

\begin{figure}[!ht]
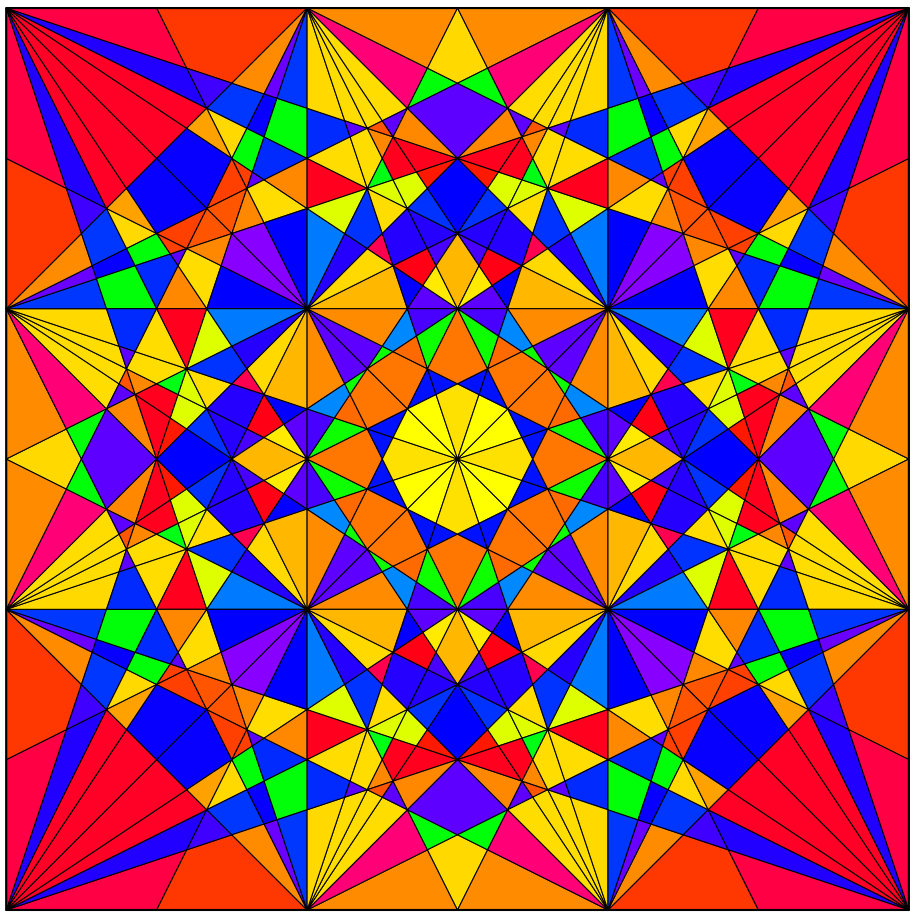

	\begin{center}

	\end{center}
	 \caption{The graph $LC(3,3)$. There are $405$ nodes and $624$ cells.}\label{FigLC33}
%\caption{ Random seed=54321 | H=5 Dev=20 Mult=1.2 | H=35 Dev=20 Mult=1.2 | H=50 Dev=5 Mult=0.5 | H=120 Dev=5 Mult=0.2 | H=240 Dev=15 Mult=2}
\end{figure}

%%%%%%%%%%%%%%%%%%%%%%%%%%%%%%%%%%%%%%
% Section LC
%%%%%%%%%%%%%%%%%%%%%%%%%%%%%%%%%%%%%%
\section{The graphs \texorpdfstring{$LC(m,n)$}{LC(m,n)}.}\label{SecLC}

The graph $LC(m,n)$ was defined in \textsection\ref{SecBCmn}.
We take an $(m+1) \times (n+1)$ square grid of nodes,
draw a chord between \emph{every} pair of grid nodes, and extend all the chords until they meet 
the boundary of the grid.
These graphs were discussed by Mustonen \cite{Mus08, Mus09, Mus10}.
Figure~\ref{Fig23Cfn} shows $LC(3,2)$ (take the black, red, and blue lines), and
Fig.~\ref{FigLC33} shows our stained glass coloring of $LC(3,3)$.

The numbers of nodes $\sN_{LC}(m,n)$  and cells $\sC_{LC}(m,n)$ in $LC(m,n)$
are given for $m, n \le 8$ in \seqnum{A333284} and \seqnum{A333282}, respectively, 
and the initial terms are shown
in Table~\ref{TabLC1}. 
Again the first row and column are the same as in Table~\ref{TabBC1}.
Mustonen~\cite[Table~3]{Mus09} gives the
first 29 terms of the diagonal sequence $\sN_{LC}(n,n)$ (\seqnum{A333285}).

% LC table start \seqnum{A333284} and \seqnum{A333282}, diagonal is \seqnum{A333285}
\begin{table}[!htb]
\caption{Numbers of nodes $\sN_{LC}(m,n)$ and cells $\sC_{LC}(m,n)$ in $LC(m,n)$ for $1 \le m, n \le 7$.}\label{TabLC1}
$$
\begin{array}{|c|c|c|c|c|c|c|c|}
\hline 
m \backslash \, n & 1 & 2 & 3 & 4 & 5 & 6 & 7  \\
\hline 
1 & 5, 4& 13, 16& 35, 46& 75, 104& 159, 214& 275, 380& 477, 648 \\
2 & 13, 16& 37, 56& 129, 192& 289, 428& 663, 942& 1163, 1672& 2069, 2940 \\
3 & 35, 46& 129, 192& 405, 624& 933, 1416& 2155, 3178& 3793, 5612& 6771, 9926 \\
4 & 75, 104& 289, 428& 933, 1416& 2225, 3288& 5157, 7520& 9051, 13188& 16129, 23368 \\
5 & 159, 214& 663, 942& 2155, 3178& 5157, 7520& 11641, 16912& 20341, 29588& 36173, 52368 \\
6 & 275, 380& 1163, 1672& 3793, 5612& 9051, 13188& 20341, 29588& 35677, 51864& 63987, 92518 \\
7 & 477, 648& 2069, 2940& 6771, 9926& 16129, 23368& 36173, 52368& 63987, 92518& 114409, 164692 \\
\hline
\end{array}
$$
\end{table}
% LC table end

For this problem we  can give  only an upper bound on the
number of nodes counted with multiplicity.  The argument does, however, avoid the use of
Sylvester's Theorem. Consider four points chosen from the $(n+1) \times (n+1)$ 
grid points, with no three points collinear. If the points form a triangle with a point in 
the interior, joining the three vertices of the triangle to the interior point and 
then extending these chords until they meet the sides of the
triangle (something we were not allowed to do in the previous case) will produce three potentially new nodes.
If the four points form a convex quadrilateral, there are also potentially three nodes that 
could be created: the intersection of the two diagonals, 
and the two points where pairs of opposite sides meet when extended.
Figure~\ref{FigSylv} shows the two cases.  The black nodes are the four grid points
and the red nodes are the potential new nodes. Of course in the second case 
the two external red points may be outside the grid (or at infinity), and so would not be counted.

%%%%%%%%%%%%%%%%%%%%%%%%%%%%%%%%%%%%%%%%%%%%%%%
%%% FIG 20
%%%%%%%%%%%%%%%%%%%%%%%%%%%%%%%%%%%%%%%%%%%%%%%
%%%%%%%%%%%%%%%%%%%%%%%
% the triangle and quad figure 
%%%%%%%%%%%%%%%%%%%%
\begin{figure}[!ht]
\begin{center}
\begin{tikzpicture}[scale=1.0]
\coordinate (P0) at (0,0);
\coordinate (P1) at (2,0);
\coordinate (P2) at (1.7,1.7);
\coordinate (P3) at (1.233,0.567);
\draw[fill=black] (0,0) circle[radius = 0.1];
\draw[fill=black] (2,0) circle[radius = 0.1];
\draw[fill=black] (1.7,1.7) circle[radius = 0.1];
\draw[fill=black] (1.233,0.567) circle[radius = 0.1];
\draw (P0) -- (P1);
\draw (P1) -- (P2);
\draw (P0) -- (P2);
\draw (-0.4,-0.184) -- (2.3, 1.058);
\draw[fill=red] (1.848,0.850) circle[radius = 0.1];
\draw (0.766,-0.566) -- (1.84,2.04);
\draw[fill=red] (1,0.0) circle[radius = 0.1];
\draw (2.537,-0.397) -- (0.3126,1.247);
\draw[fill=red] (0.8495,0.8505) circle[radius = 0.1];

\def\s{4}
\coordinate (P0) at (\s+0,0);
\coordinate (P1) at (\s+2,0);
\coordinate (P2) at (\s+2,1);
\coordinate (P3) at (\s+1,2);
\coordinate (P4) at (\s+1.6,0.8);
\coordinate (P5) at (\s+3,0);
\coordinate (P6) at (\s+2,4);
\draw[fill=black] (P0) circle[radius = 0.1];
\draw[fill=black] (P1) circle[radius = 0.1];
\draw[fill=black] (P2) circle[radius = 0.1];
\draw[fill=black] (P3) circle[radius = 0.1];
\draw[fill=red] (P4) circle[radius = 0.1];
\draw[fill=red] (P5) circle[radius = 0.1];
\draw[fill=red] (P6) circle[radius = 0.1];
\draw (\s+-0.2,-0.4) -- (\s+2.2,4.4);
\draw (\s+-0.4,0) -- (\s+3.6,0);
\draw (\s+-.4,-.2) -- (\s+2.6,1.3);
\draw (\s+.6,2.4) -- (\s+3.4,-.4);
\draw (\s+.8,2.4) -- (\s+2.2,-.4);
\draw (\s+2,-.5) -- (\s+2,4.4);
\end{tikzpicture}
\end{center}
\caption{The two possibilities for choosing four noncollinear points from an $m \times n$ grid.}
\label{FigSylv}
\end{figure}
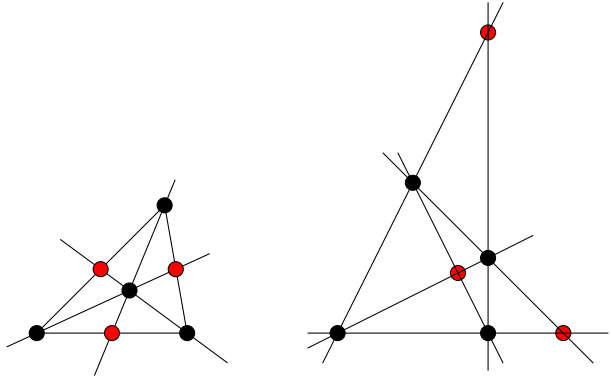

In any case, the maximum number of new nodes that are created is at most $3\,\binom{(n+1)^2}{4}
~\sim~ n^8/8$, and this is an upper bound on $\sN_{LC}(n,n)$.
This is an over-count, both because we do not always get three new nodes for each $4$-tuple of grid points, 
and because multiple intersection points are counted multiple times.
Based on his data for $n \le 29$, Mustonen~\cite{Mus09} 
makes an empirical estimate that %observes that it appears that 
$\sN_{LC}(n,n) ~\sim~ C n^8$, where $C$ is about~$0.0075$. So our constant, $1/8$ is, unsurprisingly, 
an over-estimate.

We conclude that as we progress from $BC(n,n)$ to $AC(n,n)$ to $LC(n,n)$,
the graphs become progressively more dense, and so counting the nodes with multiplicity 
gives a steadily weaker upper bound on their number.

%%%%%%%%%%%%%%%%%%%%%%%%%%%%%%%%%%%%%%
% Section: Coloring
%%%%%%%%%%%%%%%%%%%%%%%%%%%%%%%%%%%%%%

\section{Choosing the colors.}\label{SecColor}
We used three different coloring schemes.

\subsection{Number-of-sides coloring.}\label{SubSecNOSC}
The simplest scheme colors the cells according to the number of sides, with randomly chosen colors.
This is used in Fig.~\ref{FigBC42} and in figures in \cite{OEIS} (entries \seqnum{A333282}, \seqnum{A335701}, for example)
when studying the distribution of cells according to number of sides.

\subsection{The yellow and red palettes.}\label{SubSecYR}

% The area of the polygon divided by the area of n-sided regular polygon with same circumscribed radius. Taking the square root of this determines the color from the palette.
This is a  refinement of the previous scheme, which modifies the color according to the shape of the cell.
For Figs.~\ref{FigBC22}, \ref{FigBC12}, \ref{FigBC13},
\ref{FigBC14}, \ref{FigIT3},   \ref{FigIT4} the cells are either triangles
or quadrilaterals, and we use colors which darken as the cell becomes more irregular. 
More precisely, the cells are colored according to the following rule.
If the cell has $n$ sides (where $n$ is $3$ or $4$), let $\lambda$ be the area of
the cell divided by the area of an $n$-sided regular polygon with  the same circumradius.
Then the cell is assigned color number $\sqrt{\lambda}$ from the following  palettes:

 \definecolor{d3}{RGB}{150, 95, 0}
 \definecolor{e3}{RGB}{255, 255, 0}
 \definecolor{d4}{RGB}{100, 0, 0}
 \definecolor{e4}{RGB}{255, 0, 0}
 \definecolor{d5}{RGB}{30, 30, 100}
 \definecolor{e5}{RGB}{30, 30, 255}
\begin{center}
\begin{tabular}{cc}

%3-gons&
triangles&
\begin{tikzpicture}
\node [rectangle, left color=d3, right color=e3, minimum width=70pt, minimum height=1.0cm] (box) at (0,0){};
\end{tikzpicture}
\\

%4-gons&
quadrilaterals&
\begin{tikzpicture}
\node [rectangle, left color=d4, right color=e4, minimum width=70pt, minimum height=1.0cm] (box) at (0,0){};
\end{tikzpicture}
\\

\end{tabular}
\end{center}

\subsection{Random colorings.}\label{SubSecRan}
For Figs.~\ref{FigK23}, \ref{FigBC14}, \ref{FigBC33}, etc.
the color of a cell is assigned by first computing the average distance of the nodes
of the cell from the center of the picture.
These average  distances are then grouped into a certain number of bins (we used $1000$ bins),
and the nonempty bins are assigned a random color from the standard spectrum 
from red to violet.
This ensures a symmetrical coloring with contrasting colors for neighboring cells.
In practice we  do this several times and then choose the most appealing picture.
We also have the option of restricting the color palette to achieve certain effects (reds, blues, 
and greens for a cathedral-like window, or various shades of browns for
the frames that we will see in Part 2).

%%%%%%%%%%%%%%%%%%%%%%%%%%%%%%%%%%%%%%
% Section 2: Triangles (1) This will be part of Part 2 of the paper.
%
%%%%%%%%%%%%%%%%%%%%%%%%%%%%%%%%%%%%%%
%\section{Triangles (1)}\label{SecT1}
% [Latt 118 p 68ff]
%The graph $\sA(n)$ is an $n$-reticulated triangle  (that is, each side of the triangle contains $n-1$ equally
 %spaced nodes and two of the vertices), and each vertex is joined by chords to
%all the nodes on the opposite side.  Nodes on distinct sides are not joined.

\section{Acknowledgments}
We thank Max Alekseyev, Gareth McCaughan, Ed Pegg, Jr.,  and Jinyuan Wang for their assistance during the course of this work.
Tom Duff and  Keith F. Lynch carried out extensive computations to verify the applicability
of Sylvester's theorem (see \textsection\ref{SecAC}).
We made frequent use of the \emph{gfun} Maple program \cite{GFUN}
and the \emph{TikZ} Latex package \cite{Cre11, Tan10}.

%%%%%%%%%%%%%%%%%%%%%%%%%%%%%%%%%%%%%%%%%%%%

\bigskip
\hrule
\bigskip

\noindent 2010 Mathematics Subject Classification 05A16, 05C10, 05C30, 52B05

\begin{thebibliography}{99}

\bibitem{Alex10} % used CC
M. A. Alekseyev,
 On the number of two-dimensional threshold functions,
 \emph{SIAM J. Discr. Math.}, \textbf{24:4} (2010), 1617--1631.
 
\bibitem{ABZ15} % used CC
M. A. Alekseyev, M. Basova, and N. Yu. Zolotykh,
 On the minimal teaching sets of two-dimensional threshold functions,
 \emph{ SIAM J. Discr. Math.}, \textbf{29:1} (2015), 157--165.
 
 \bibitem{Rose2} % used CC
 L. Blomberg, S. R. Shannon, and N. J. A. Sloane,
 Graphical enumeration and stained glass windows, 2: Polygons, frames, crosses, etc.,
 in preparation, 2020.

\bibitem{Bela} % used CC
B. Bollob\'{a}s, 
\emph{Graph Theory: An Introductory Course}, Springer, 1979.

\bibitem{RLG1} % used CC
F. Chung and R. Graham,
Primitive juggling sequences,
\emph{Amer. Math. Monthly}, \textbf{115:3} (2008), 185--194.

\bibitem{Cre11} % used CC
J. Cr\'{e}mer,
A very minimal introduction to \emph{TikZ},
March 11, 2011; 
%available electronically from
\url{https://cremeronline.com/LaTeX/minimaltikz.pdf}.

\bibitem{Fre75} % used CC
J. W. Freeman,
The number of regions determined by a convex polygon,
\emph{Math. Mag.}, \textbf{49:1} (1975), 23--26.

\bibitem{Gri10} % used CC
M. Griffiths, 
Counting the regions in a regular drawing of $K_{n,n}$, 
\emph{J. Integer Sequences}, \textbf{13} (2010), \#10.8.5.

\bibitem{Hall04} % used CC
H. T. Hall, 
\emph{Counterexamples in Discrete Geometry},
PhD Dissertation, Mathematics Department, University of California Berkeley, 2004.

\bibitem{Harary} % used CC
F. Harary,
\emph{Graph Theory}, Addison-Wesley, Reading MA, 1969.

\bibitem{Leg09} % used CC
S. Legendre, 
The number of crossings in a regular drawing of the complete bipartite graph, 
\emph{J. Integer Sequences}, \textbf{12} (2009), \#09.5.5.

\bibitem{Mus08} % used CC  See A333282
S. Mustonen, 
Statistical accuracy of geometric constructions, September 2, 2008; 
\url{http://www.survo.fi/papers/GeomAccuracy.pdf}.

\bibitem{Mus09} % used  CC   see A333282
S. Mustonen, 
On lines and their intersection points in a rectangular grid of points, April 16, 2009;
\url{http://www.survo.fi/papers/PointsInGrid.pdf}.

\bibitem{Mus10} % used CC   see A333282, A333285
S. Mustonen, 
On lines going through a given number of points in a rectangular grid of points, May 12, 2010;
\url{http://www.survo.fi/papers/LinesInGrid2.pdf}.

\bibitem{OEIS} % used CC
The OEIS Foundation Inc., 
\emph{The On-Line Encyclopedia of Integer Sequences}, 2020;  
%\tt{https://oeis.org}.  \rm 
\url{https://oeis.org}.

\bibitem{PfZi04} % used CC
M. E. Pfetsch and G. M. Ziegler, 
Large chambers in a lattice polygon,   December 13, 2004;
\url{http://www.mathematik.tu-darmstadt.de/~pfetsch/chambers}.

\bibitem{Pfi89} % used CC
R. E. Pfiefer,
The historical development of J.~J.~Sylvester's Four Point Theorem,
\emph{Math. Mag.}, \textbf{62:5} (1989), 309--317.

\bibitem{PoRu98} % used CC
B. Poonen and M. Rubinstein, 
The number of intersection points made by the diagonals of a regular polygon, 
\emph{SIAM J. Discr. Math.}, \textbf{11:1} (1998), 135--156.

\bibitem{GFUN} % used CC
B. Salvy and P. Zimmermann,
GFUN: a Maple package for the manipulation of generating and holonomic functions in one variable,
\emph{ACM Trans.  Math. Software},
{\textbf 20} (1994), 163--177.

\bibitem{SZ98} % used CC
V. N. Shevchenko and N. Yu. Zolotykh, 
On the complexity of deciphering the threshold
functions of $k$-valued logic, (Russian)
\emph{Dokl. Akad. Nauk}, \textbf{362:5} (1998), 606--608;
(English translation) \emph{Dokl. Math.}, \textbf{58} (1998), 268--270.

\bibitem{Slo10} % used CC
N. J. A. Sloane,
The email servers and Superseeker, 2010;
\url{https://oeis.org/ol.html}.
 
\bibitem{Sol78} % used CC
H. Solomon,
\emph{Geometric Probability}, SIAM, Philadelphia, 1978.

\bibitem{Som98} % used CC
S. E. Sommars and T. Sommars,
Number of triangles formed by intersecting diagonals of a regular polygon,
\emph{J. Integer Sequences}, \textbf{1} (1998), \#98.1.5.

\bibitem{Tan10} % used CC
T. Tantau,
\emph{The PGF/TikZ Programming Language},
Version 2.10, CTAN Org., October 25 2010.

\bibitem{Tut63} % used  CC  See A000139 for a link
W. T. Tutte, 
A census of planar maps, 
\emph{Canad. J. Math.}, \textbf{15} (1963), 249--271.

\bibitem{Tut68} % used CC  See A000184 for a link
W. T. Tutte, 
On the enumeration  of planar maps, 
\emph{Bull. Amer. Math. Soc.}, \textbf{74} (1968), 64--74.

\bibitem{Zol01}  % used CC
N. Yu. Zolotykh, 
On the complexity of deciphering threshold functions in two variables, (Russian),
in \emph{Proc. 11th Internat. School Seminar ``Synthesis and complexity of control systems,'' Part I}, Center of Applied Research, Moscow State Univ. Faculty of Mechanics and Mathematics,
Moscow, Russia, 2001, pp. 74--79.

\end{thebibliography}
\end{document}